\documentclass[12pt,a4paper]{article}
\usepackage[utf8]{inputenc}
\usepackage{amsmath, amssymb, amsthm, verbatim, hyperref}
\usepackage{mathrsfs}
\usepackage[dvipsnames]{xcolor}
\usepackage{ragged2e}
\usepackage{marginnote}
\usepackage{enumerate}
\usepackage{makecell}
\usepackage{longtable}
\usepackage{epsfig}
\usepackage{tikz}
\usepackage{ytableau}
\usepackage{hyperref}
\usepackage{caption}
\usepackage{subcaption}
\hypersetup{hidelinks}
\usepackage{fancyvrb}
\usepackage[left=2.50cm, right=2.50cm, top=2.00cm, bottom=2.00cm]{geometry}
\usetikzlibrary {arrows.meta}

\theoremstyle{plain}

\newtheorem{innercustomthm}{Theorem}
\newenvironment{repthm}[1]
  {\renewcommand\theinnercustomthm{#1}\innercustomthm}
  {\endinnercustomthm}

\newtheorem{innercustomcor}{Corollary}
\newenvironment{repcor}[1]
  {\renewcommand\theinnercustomcor{#1}\innercustomcor}
  {\endinnercustomthm}

\newtheorem{thm}{Theorem}[section]
\newtheorem{prop}[thm]{Proposition}
\newtheorem{cor}[thm]{Corollary}
\newtheorem{lemma}[thm]{Lemma}
\newtheorem{conjecture}[thm]{Conjecture}

\theoremstyle{definition}
\newtheorem{definition}[thm]{Definition}
\theoremstyle{remark}

\newtheorem{remark}[thm]{Remark}

\newcommand{\ZZ}{\mathbb{Z}}

\newcommand{\RR}{\mathbb{R}}

\newcommand{\BB}{\mathcal{B}}
\newcommand{\BBB}{\mathfrak{B}}
\newcommand{\AAA}{\mathscr{A}}
\newcommand{\DDD}{\mathfrak{D}}
\newcommand{\PPP}{\mathfrak{P}}
\newcommand{\SSS}{\mathfrak{S}}
\newcommand{\des}{\operatorname{des}}
\newcommand{\Ass}{\textsf{Ass}}

\title{Poset Associahedra and Stack-sorting}
\author{Son Nguyen, Andrew Sack\thanks{This material is based upon work supported by the National Science Foundation Graduate Research Fellowship Program under Grant No. DGE-2034835 and National Science Foundation Grants No. DMS-1954121 and DMS-2046915. Any opinions, findings, and conclusions or recommendations expressed in this material are those of the author(s) and do not necessarily reflect the views of the National Science Foundation.}}

\date{\vspace{-2em}}

\allowdisplaybreaks

\begin{document}
\ytableausetup{centertableaux}

\maketitle

    \begin{abstract}
    For any finite connected poset $P$, Galashin introduced a simple convex $(|P|-2)$-dimensional polytope $\AAA(P)$ called the poset associahedron. For a certain family of posets, whose poset associahedra interpolate between the classical permutohedron and associahedron, we give a simple combinatorial interpretation of the $h$-vector. Our interpretation relates to the theory of stack-sorting of permutations. It also allows us to prove real-rootedness of some of their $h$-polynomials.
    \end{abstract}

\tableofcontents


\section{Introduction}\label{sec:intro}

    For a finite connected poset $P$, Galashin introduced the \textit{poset associahedron} $\AAA(P)$ (see \cite{galashin2021poset}). The faces of $\AAA(P)$ correspond to {\it tubings} of $P$, and the vertices of $\AAA(P)$ correspond to {\it maximal tubings} of $P$;  see Section~\ref{sec:poset-ass} for the definitions. $\AAA(P)$ can also be described as a compactification of the configuration space of order-preserving maps $P \rightarrow \RR$. Many polytopes can be described as poset associahedra, including permutohedra and associahedra. In particular, when $P$ is the claw poset, i.e. $P$ consists of a unique minimal element $0$ and $n$ pairwise-incomparable elements, then $\AAA(P)$ is the $n$-permutohedron. On the other hand, when $P$ is a chain of $n+1$ elements, i.e. $P = C_{n+1}$, then $\AAA(P)$ is the associahedron $K_{n+1}$. Among many different combinatorial interpretations for the $h$-vector $(h_0,h_1,\ldots,h_{n-1})$ of $K_{n+1}$, we want to recall the following interpretation: $h_i$ counts the number of 231-avoiding permutations with exactly $i$ descents. 

    Stack-sorting is a function $s: \SSS_n \rightarrow \SSS_n$ which attempts to sort the permutations $w$ in $\SSS_n$ in linear time, not always sorting them completely (see definition in Section \ref{subsec:stack-sorting}). A permutation $w\in S_n$ is stack-sortable if $s(w) = 12\ldots n$. It is well-known that stack-sortable permutations are exactly 231-avoiding permutations. Thus, we have an alternative interpretation for the $h$-vector of $\AAA(C_{n+1})$: $h_i$ counts the number of permutations in $s^{-1}(12\ldots n)$ with exactly $i$ descents.

    The focus of our paper is the posets $A_{n,k} = C_{n+1} \oplus A_k$ where $A_k$ is the antichain of $k$ elements. In particular, $A_{0,k}$ is a claw poset, and $A_{n,0}$ is the chain $C_{n+1}$. Surprisingly, the $h$-vector of $\AAA(A_{n,k})$ is also counted by descents of stack-sorting preimages. Let $\SSS_{n,k} = \{w~|~w\in \SSS_{n+k}, w_i = i~\text{for all}~i>k\}$, we prove the following generalization of the above classic result.

    \begin{repthm}{\ref{thm:A_n,k-h-vector}}
        Let $h = (h_0, h_1, \ldots, h_{n+k-1})$ be the $h$-vector of $\AAA(A_{n,k})$. Then $h_i$ counts the number of permutations in $s^{-1}(\SSS_{n,k})$ with exactly $i$ descents.
    \end{repthm}

    An immediate corollary of Theorem \ref{thm:A_n,k-h-vector} is $\gamma$-nonnegativity of $\AAA(A_{n,k})$. In particular, we have the following result by Br\"anden.

    \begin{repthm}{\ref{thm:ss-preimage-gamma}}[\cite{branden2008actions}]
        For $A \subseteq \SSS_n$, we have
        \[ \sum_{\sigma \in s^{-1}(A)} x^{\des(\sigma)} = \sum_{m = 0}^{\lfloor \frac{n-1}{2}\rfloor} \dfrac{|\{ \sigma \in s^{-1}(A)~:~\text{peak}(\sigma) = m \}|}{2^{n-1-2m}} x^m(1+x)^{n-1-2m}, \]
        where $\text{peak}(\sigma)$ is the number of index $i$ such that $\sigma_{i-1} < \sigma_i > \sigma_{i+1}$.
    \end{repthm}

    Thus, we have the following corollary.

    \begin{repcor}{\ref{cor:A_n,k-gamma-nonnegative}}
        The $\gamma$-vector of $\AAA(A_{n,k})$ is nonnegative.
    \end{repcor}

    In addition, in the process of proving Theorem \ref{thm:A_n,k-h-vector}, we find the size of $s^{-1}(\SSS_{n,k})$ in terms of $k!$ and the Catalan convolution $C_n^{(k)}$, which will be introduced in Section \ref{subsec:catalan-convolution}.

    \begin{repcor}{\ref{cor:ss-size}}
        For all $n,k \geq 0$, we have
        \[ |s^{-1}(\SSS_{n,k})| = k! \cdot C_n^{(k)}. \]
    \end{repcor}

    Note that $C_n^{(0)}$ is the classic Catalan number $C_n$. Thus, setting $k=0$ in Corollary \ref{cor:ss-size}, we obtain the classic result that $s^{-1}(12\ldots n) = C_n$. Finally, in Section \ref{sec:real-rootedness}, we will use a ``happy coincidence'' in stack-sorting to prove real-rootedness of the $h$-polynomials of $\AAA(A_{n,2})$.

    \begin{repthm}{\ref{thm:A_2-real-rooted}}
        Let $H_n(x)$ be the $h$-polynomial of $\AAA(A_{n,2})$. Then, $H_n(x)$ is real-rooted.
    \end{repthm}

\section{Definition}\label{sec:definition}

\subsection{Polytope and face numbers}\label{subsec:face_numbers}

    A \textit{convex polytope} $P$ is the convex hull of a finite collection of points in $\RR^n$. The \textit{dimension} of a polytope is the dimension of its affine span. A face $F$ of a convex polytope $P$ is the set of points in $P$ where some linear functional achieves its maximum on $P$. Faces that consist of a single point are called \textit{vertices} and $1$-dimensional faces are called \textit{edges} of $P$. A $d$-dimensional polytope $P$ is \textit{simple} if any vertex of $P$ is incident to exactly $d$ edges.

    For a $d$-dimensional polytope $P$, the \textit{face number} $f_i(P)$ is the number of $i$-dimensional faces of $P$. In particular, $f_{0}(P)$ counts the vertices and $f_1(P)$ counts the edges of $P$. The sequence $(f_{0}(P),f_1(P),\ldots,f_{d}(P))$ is called the \textit{$f$-vector} of $P$, and the polynomial
    \[ f_P(t) = \sum_{i=0}^{d} f_i(P) t^{i} \]
    is called the \textit{$f$-polynomial} of $P$. The \textit{$h$-vector} $(h_0(P),\ldots,h_d(P))$ and \textit{$h$-polynomial} $h_P(t) = \sum_{i=0}^dh_i(P)t^i$ are defined by the relation
    \[ f_P(t) = h_P(t+1). \]
    It is well-known that when $P$ is a simple polytope, its $h$-vector satisfies the Dehn-Sommerville symmetry: $h_i(P) = h_{d-i}(P)$. When the $h$-polynomial is symmetric, it has a unique expansion in terms of binomials $t^i(1+t)^{d-2i}$ for $0\leq i \leq d/2$. This unique expansion gives the \textit{$\gamma$-vector} $(\gamma_0(P),\ldots,\gamma_{\lfloor\frac{d}{2}\rfloor}(P))$ and \textit{$\gamma$-polynomial} $\gamma_P(t) = \sum_{i = 0}^{\lfloor\frac{d}{2}\rfloor}\gamma_i(P)t^i$ defined by 
    \[ h_P(t) = \sum_{i=0}^{\lfloor\frac{d}{2}\rfloor}\gamma_i(P)t^i(1+t)^{d-2i} = (1+t)^d\gamma_P\left( \dfrac{t}{(1+t)^2} \right).\]

\subsection{Poset associahedra}\label{subsec:poset_ass}

    We start with some poset terminologies.

    \begin{definition}\label{def:poset_def}
        Let $(P,\preceq)$ be a finite poset, and $\tau,\sigma \subseteq P$ be subposets.

        \begin{itemize}
            \item $\tau$ is \textit{connected} if it is connected as an induced subgraph of the Hasse diagram of $P$.

            \item $\tau$ is \textit{convex} if whenever $x,z\in \tau$ and $y \in P$ such that $x\preceq y \preceq z$, then $y\in \tau$.

            \item $\tau$ is a \textit{tube} of $P$ if it is connected and convex. $\tau$ is a \textit{proper tube} if $1 < |\tau| < |P|$.

            \item $\tau$ and $\sigma$ are \textit{nested} if $\tau \subseteq \sigma$ or $\sigma \subseteq \tau$. $\tau$ and $\sigma$ are \textit{disjoint} if $\tau \cap \sigma = \emptyset$.

            \item We say $\sigma \prec \tau$ if $\sigma \cap \tau = \emptyset$, and there exists $x\in \sigma$ and $y\in \tau$ such that $x\preceq y$.

            \item A \textit{tubing} $T$ of $P$ is a set of proper tubes such that any pair of tubes in $T$ is either nested or disjoint, and there is no subset $\{\tau_1,\tau_2,\ldots,\tau_k\}\subseteq T$ such that $\tau_1 \prec \tau_2 \prec \ldots \prec \tau_k \prec \tau_1$. We will refer to the latter condition as the \textit{acyclic condition}.

            \item A tubing $T$ is \textit{maximal} if it is maximal under inclusion, i.e. $T$ is not a proper subset of any other tubing.
        \end{itemize}
    \end{definition}

    Figure \ref{fig:tubingEx} shows examples and non-examples of tubings of posets. Note that the right-most example in Figure \ref{subfig:tubingNonEx} is a non-example since it violates the acyclic condition. In particular, if we label the tubes from right to left as $\tau_1,\tau_2,\tau_3$, then we have $\tau_1 \prec \tau_2 \prec \tau_3 \prec \tau_1$.

    \begin{figure}[h!]
     \centering
        \begin{subfigure}[b]{0.45\textwidth}
            \centering
            \includegraphics[scale = 0.5]{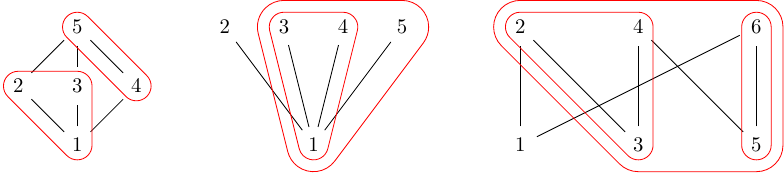}
            \caption{Examples}
            \label{subfig:tubingEx}
        \end{subfigure}
     \quad
        \begin{subfigure}[b]{0.45\textwidth}
            \centering
            \includegraphics[scale = 0.5]{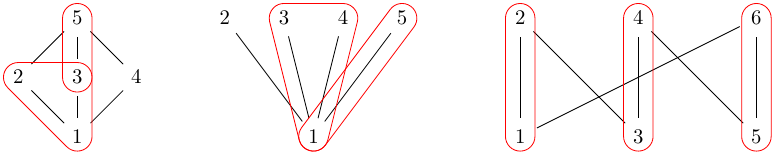}
            \caption{Non-examples}
            \label{subfig:tubingNonEx}
        \end{subfigure}

        \caption{Examples and non-examples of tubings of posets}
        \label{fig:tubingEx}
    \end{figure}

    \begin{definition}[{\cite[Theorem 1.2]{galashin2021poset}}]
        For a finite poset $P$, there exists a simple, convex polytope $\AAA(P)$ of dimension $|P|-2$ whose face lattice is isomorphic to the set of tubings ordered by reverse inclusion. The faces of $\AAA(P)$ correspond to tubings of $P$, and the vertices of $\AAA(P)$ correspond to maximal tubings of $P$. This polytope is called the \textbf{poset associahedron} of $P$.
    \end{definition}

    Examples of poset associahedra can be seen in Figure \ref{fig:posetAssEx}. In particular, if $P$ is a claw, i.e. $P$ consists of a unique minimal element $0$ and $n$ pairwise-incomparable elements as shown in Figure \ref{subfig:posetExPerm}, $\AAA(P)$ is a permutohedron. If $P$ is a chain, $\AAA(P)$ is an associahedron.

    \begin{figure}[h!]
     \centering
        \begin{subfigure}[c]{0.45\textwidth}
            \centering
            \includegraphics[scale = 0.3]{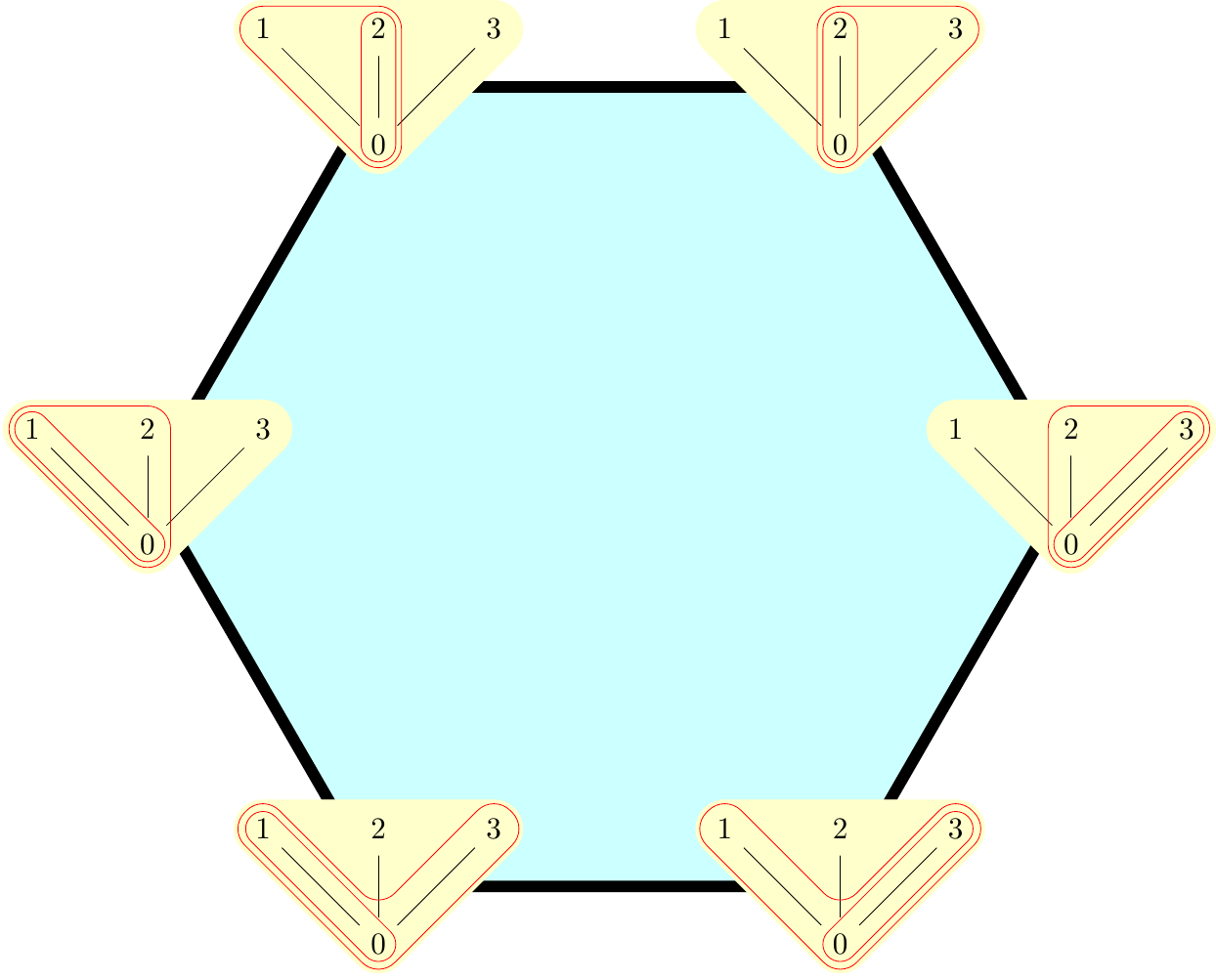}
            \caption{Permutohedron}
            \label{subfig:posetExPerm}
        \end{subfigure}
     \quad
        \begin{subfigure}[c]{0.45\textwidth}
            \centering
            \includegraphics[scale = 0.3]{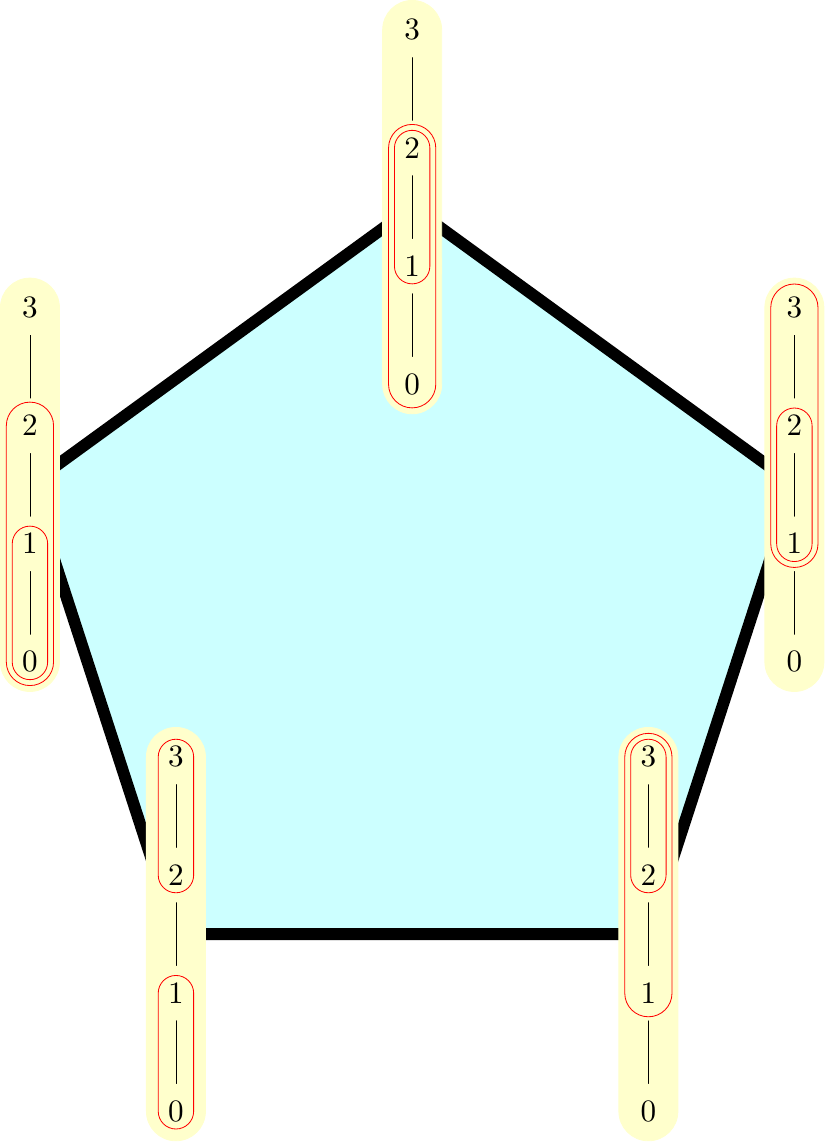}
            \caption{Associahedron}
            \label{subfig:posetExAss}
        \end{subfigure}

        \caption{Permutohedron and associahedron as poset associahedra}
        \label{fig:posetAssEx}
    \end{figure}

\subsection{Graph associahedra}\label{subsec:graph_ass}

    Graph associahedra are generalized permutohedra arising as special cases of nestohedra. We refer the readers to \cite{postnikov2006faces} for a comprehensive study of face numbers of generalized permutohedra and nestohedra.

    \begin{definition}\label{def:graph_def}
        Let $G = (V,E)$ be a connected graph, and $\tau,\sigma\subseteq V$ be subsets of vertices.

        \begin{itemize}
            \item $\tau$ is a \textit{tube} of $G$ if $\tau \neq V$ and it induces a connected subgraph of $G$.
            
            \item $\tau$ and $\sigma$ are \textit{nested} if $\tau \subseteq \sigma$ or $\sigma \subseteq \tau$. $\tau$ and $\sigma$ are \textit{disjoint} if $\tau \cap \sigma = \emptyset$.

            \item $\tau$ and $\sigma$ are \textit{compatible} if they are nested or they are disjoint and $\tau \cup \sigma$ is not a tube.

            \item A \textit{tubing} $T$ of $G$ is a set of pairwise compatible tubes.

            \item A tubing $T$ is \textit{maximal} if it is maximal by inclusion, i.e. $T$ is not a proper subset of any other tubing.
        \end{itemize}
    \end{definition}

    Figure \ref{fig:graphTubingEx} shows examples and non-examples of tubings of graphs. Note that the left-most example in Figure \ref{subfig:graphTubingNonEx} is a non-example since the tubes $\{1\}$ and $\{4\}$ are disjoint yet their union $\{1,4\}$ is still a tube. The same reason applies for the right-most example.

    \begin{figure}[h!]
     \centering
        \begin{subfigure}[b]{0.45\textwidth}
            \centering
            \includegraphics[scale = 0.5]{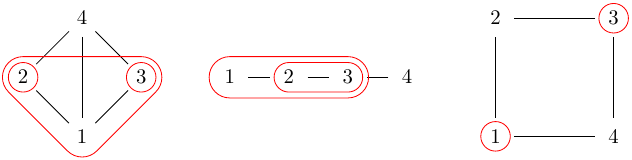}
            \caption{Examples}
            \label{subfig:graphTubingEx}
        \end{subfigure}
     \quad
        \begin{subfigure}[b]{0.45\textwidth}
            \centering
            \includegraphics[scale = 0.5]{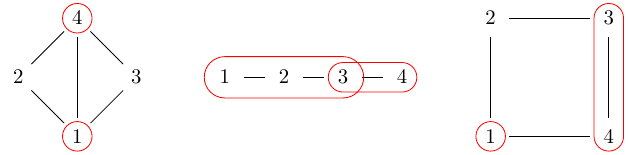}
            \caption{Non-examples}
            \label{subfig:graphTubingNonEx}
        \end{subfigure}

        \caption{Examples and non-examples of tubings of graphs}
        \label{fig:graphTubingEx}
    \end{figure}

    \begin{definition}
        For a connected graph $G = (V,E)$, the \textbf{graph associahedron} of $G$ is a simple, convex polytope $\Ass(G)$ of dimension $|V|-1$ whose face lattice is isomorphic to the set of tubings ordered by reverse inclusion. The faces of $\Ass(G)$ correspond to tubings of $G$, and the vertices of $\Ass(G)$ correspond to maximal tubings of $G$.
    \end{definition}

    Examples of graph associahedra can be seen in Figure \ref{fig:graphAssEx}. In particular, if $G$ is a complete graph, $\Ass(G)$ is a permutohedron. If $G$ is a path graph, $\Ass(G)$ is an associahedron.

    \begin{figure}[h!]
     \centering
        \begin{subfigure}[c]{0.45\textwidth}
            \centering
            \includegraphics[scale = 0.3]{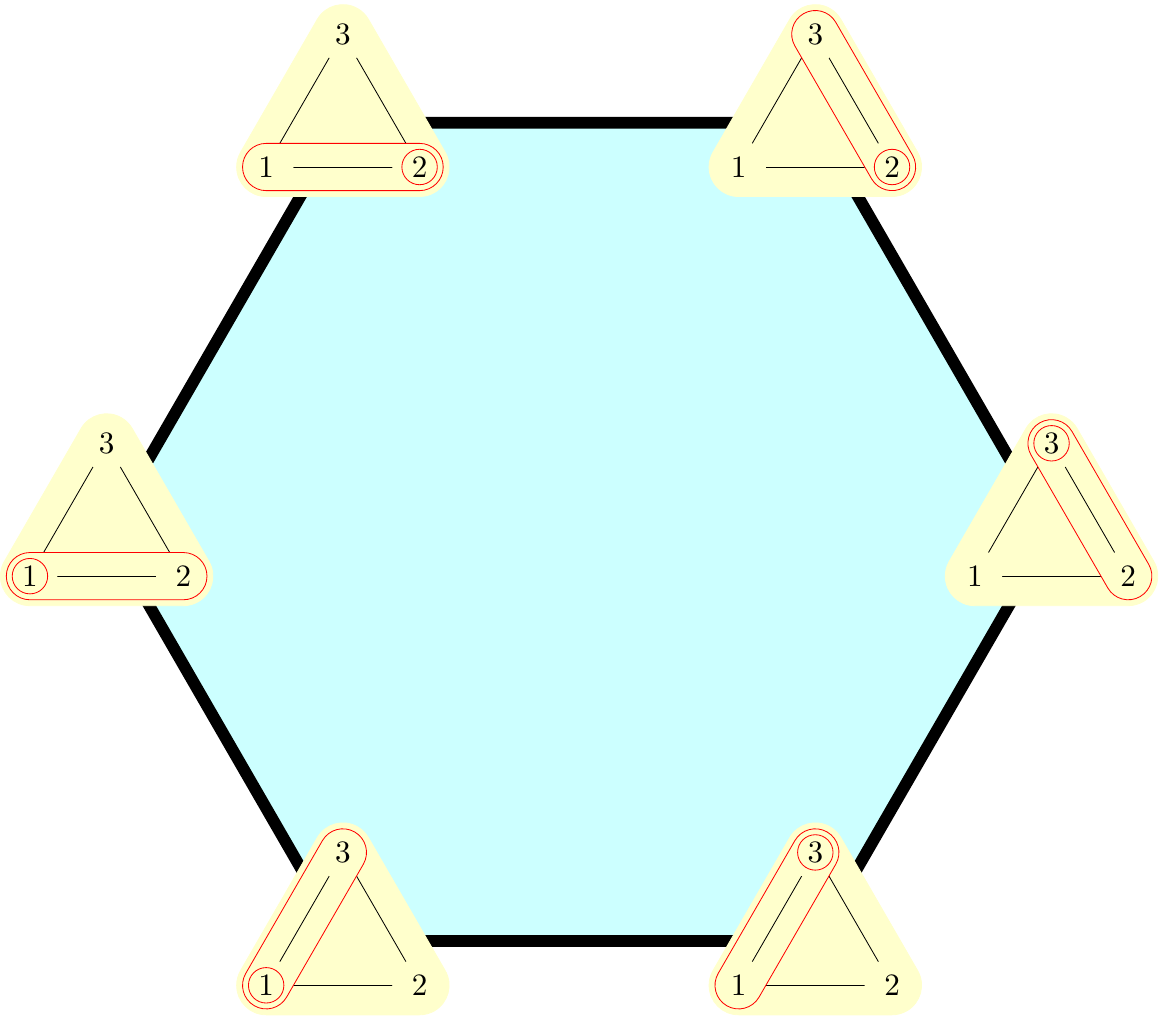}
            \caption{Permutohedron}
            \label{subfig:graphExPerm}
        \end{subfigure}
     \quad
        \begin{subfigure}[c]{0.45\textwidth}
            \centering
            \includegraphics[scale = 0.3]{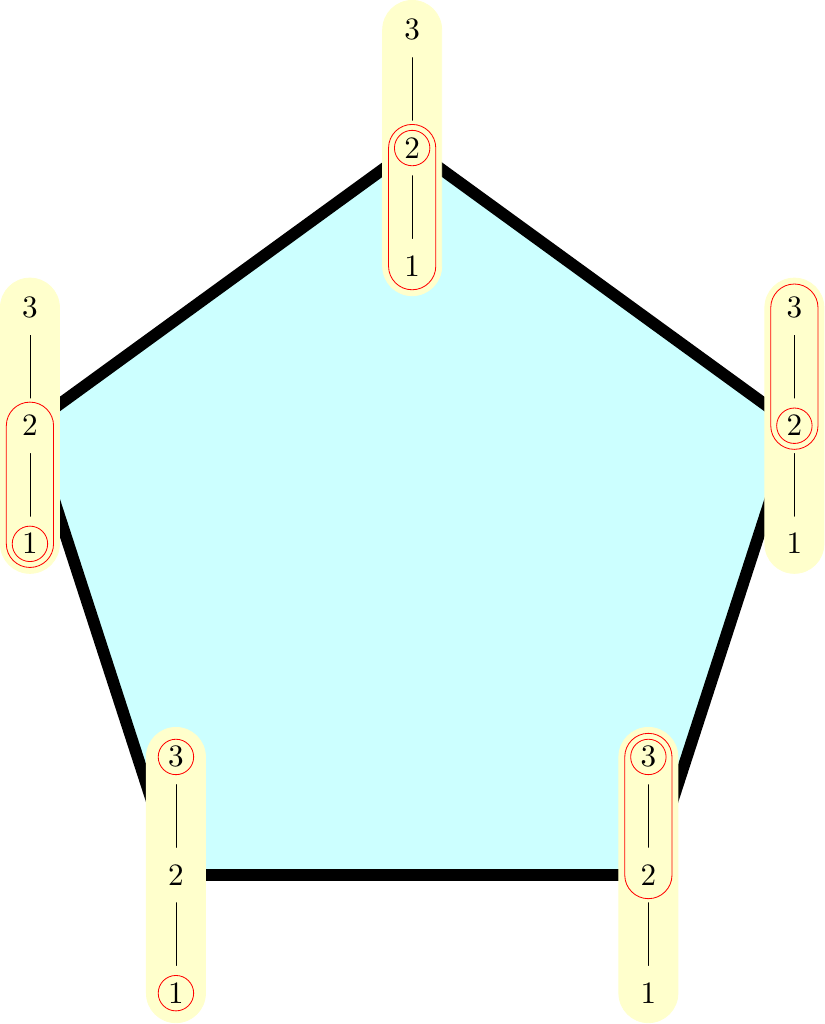}
            \caption{Associahedron}
            \label{subfig:graphExAss}
        \end{subfigure}

        \caption{Permutohedron and associahedron as graph associahedra}
        \label{fig:graphAssEx}
    \end{figure}

    Every maximal tubing of $G$ can be associated with a \textit{$\BB$-tree}. Recall that a \textit{rooted tree} is a tree with a distinguished node, called its \textit{root}. One can view a rooted tree $T$ as a partial order on its nodes in which $i <_T j$ if $j$ lies on the unique path from $i$ to the root. For a node $i$ in a rooted tree $T$, let $T_{\leq i} = \{j~|~j \leq_T i\}$ be the set of all descendants of $i$. Note that $i \in T_{\leq i}$. Nodes $i$ and $j$ in a rooted tree are called \textit{incomparable} if neither $i$ is a descendant of $j$, nor $j$ is a descendant of $i$. A \textit{descent} of $T$ is an edge $(i,j) \in E$ such that $i<j$ and $j <_T i$. We denote $\des (T)$ the number of descents in $T$.

    \begin{definition}\label{def:B-tree}
        For a maximal tubing $\BB$ of a graph $G = ([n],E)$, its \textit{$\BB$-tree} is a rooted tree $T$ on the node set $[n]$ such that
        \begin{itemize}
            \item For any $i\in[n]$ such that $i$ is not the root, one has $T_{\leq i}\in \BB$.
            \item For $k\geq 2$ incomparable nodes $i_1,\dots,i_k\in[n]$, one has $\bigcup_{j=1}^k T_{\leq i_j} \not\in \BB$.
        \end{itemize}
    \end{definition}

    Figure \ref{fig:Btree} shows three $\BB$-trees corresponding to three maximal tubings of a path graph. It is clear that $\BB$-trees of the same graph are not necessarily isomorphic.

    \begin{figure}[h!]
        \centering
        \includegraphics{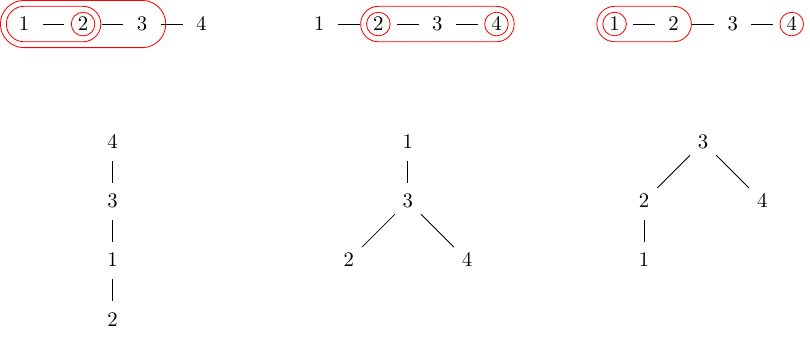}
        \caption{Maximal tubings of a path graph and their corresponding $\BB$-trees}
        \label{fig:Btree}
    \end{figure}

    The $h$-polynomial of $\Ass(G)$ is counted by the descents of the $\BB$-trees.

    \begin{thm}[{\cite[Corollary 8.4]{postnikov2006faces}}]\label{thm:b-tree-h-vector}
        For a connected graph $G$, the $h$-polynomial of $\Ass(G)$ is given by
        \[ h_{\Ass(G)}(t) = \sum_{T} t^{\des(T)}, \]
        where the sum is over all $\BB$-trees $T$.
    \end{thm}

    \subsubsection{Graph associahedra and poset associahedra}\label{subsubsec:graph_ass_poset_ass}

    Despite the similarity between graph associahedra and poset associahedra, neither of them is a subset of the other. Nevertheless, when the Hasse diagram of a poset $P$ is a tree, let $G_P$ be the line graph of the Hasse diagram of $P$, then $\AAA(P)$ is isomorphic to $\Ass(G_P)$. For instance, if $P$ is a claw, then $G_P$ is a complete graph, and $\AAA(P)$ and $\Ass(G_P)$ are both permutohedra. If $P$ is a chain, then $G_P$ is a path graph, and $\AAA(P)$ and $\Ass(G_P)$ are both associahedra. One can see a clear correspondence betweem tubings of $P$ and $G_P$ in Figures \ref{fig:posetAssEx} and \ref{fig:graphAssEx}.

    \subsection{Stack-sorting}\label{subsec:stack-sorting}

    First introduced by Knuth in \cite{knuth1973art}, the \textit{stack-sorting algorithm} led to the study of \textit{pattern avoidance} in permutations. In \cite{west1990permutations}, West defined a deterministic version of Knuth's stack-sorting algorithm, which we call the \textit{stack-sorting map} and denote by $s$. The stack-sorting map is defined as follows.

    \begin{definition}[Stack-sorting]\label{def:stack-sorting}
        Given a permutation $\pi \in \SSS_n$, $s(\pi)$ is obtained through the following procedure. Iterate through the entries of $\pi$. In each iteration,
        \begin{itemize}
            \item if the stack is empty or the next entry is smaller than the entry at the top of the stack, push the next entry to the top of the stack;
            \item else, pop the entry at the top of the stack to the end of the output permutation.
        \end{itemize}
    \end{definition}

    Figure \ref{fig:ss-ex} illustrates the stack-sorting process on $\pi = 3142$.

    \begin{figure}[h!]
        \centering
        \includegraphics[scale=0.5]{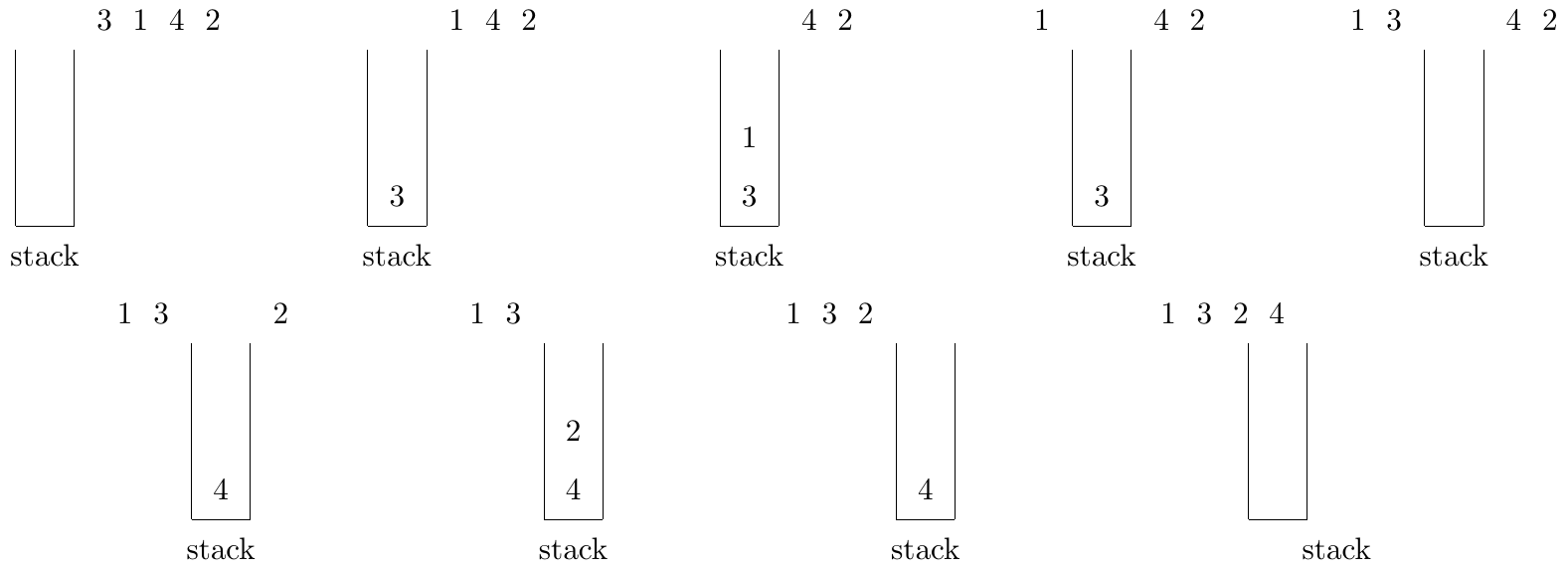}
        \caption{Example of $s(3142)$}
        \label{fig:ss-ex}
    \end{figure}

    Another way to define $s$ is by \textit{decreasing binary trees}. Recall that a \textit{binary tree} is a rooted tree in which each node has at most 2 children, usually called the left and right child. A decreasing binary tree is a binary tree whose $n$ nodes have been labeled bijectively with the numbers $\{1,2,\ldots,n\}$, such that the number in each node is larger than the numbers in its children. 

    There is a natural bijection between decreasing binary trees of size $n$ and permutations in $\SSS_n$ by \textit{inorder reading}. To read a binary tree in inorder, first we read the left subtree in inorder. Then we read the root, and finally we read the right subtree in inorder. Note that this is a recursive definition. For a decreasing binary tree $T$, we denote by $\mathcal{I}(T)$ the permutation obtained by reading $T$ in inorder. Recall that a \textit{descent} of a permutation $w$ is an index $i$ such that $w_i > w_{i+1}$. Notice that for every decreasing binary tree $T$, the descents of $\mathcal{I}(T)$ are in one-to-one correspondence with the right edges of $T$.

    Another order to read a binary tree is \textit{postorder}. To read a binary tree in postorder, first we read the left subtree in postorder. Then we read the right subtree in order before we read the root. This is also a recursive definition. For a decreasing binary tree $T$, we denote by $\mathcal{P}(T)$ the permutation obtain by reading $T$ in postorder.

    \begin{figure}[h!]
        \centering
        \includegraphics[scale=1]{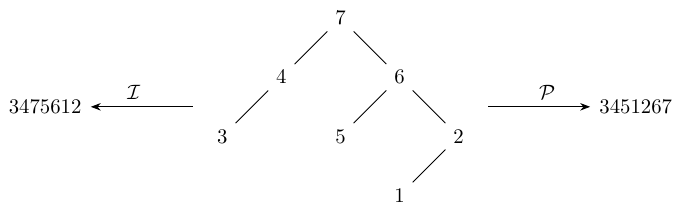}
        \caption{Reading a binary tree in inorder and postorder }
        \label{fig:reading-order-ex}
    \end{figure}

    Figure \ref{fig:reading-order-ex} shows two permutations obtained by reading a binary tree in inorder and postorder. The reading orders of binary trees give an alternate definition of stack-sorting.

    \begin{prop}[{\cite[Corollary 8.26]{bona2022combinatorics}}]\label{prop:stacksorting-postorder}
        For any $\pi \in \SSS_n$, one has
        \[ s(\pi) = \mathcal{P}(\mathcal{I}^{-1}(\pi)). \]
    \end{prop}

    For example, we encourage the the readers to check that $s(3475612) = 3451267$, which matches the example in Figure \ref{fig:reading-order-ex}.

    \subsection{Catalan convolution}\label{subsec:catalan-convolution}

    The \textit{Catalan numbers}, $C_n = \frac{1}{n+1}\binom{2n}{n}$, are one of the most well-known sequences in combinatorics. Among hundreds of objects counted by the Catalan numbers, three well-known objects are binary trees, \textit{stack-sortable permutations}, and \textit{Dyck paths}.

    A permutation $w\in \SSS_n$ is stack-sortable if $s(w) = 12\ldots n$. Given a binary tree $T$, there is only one way to label the nodes of $T$ such that the postorder reading permutation is the identity permutation. This gives a natural bijection between stack-sortable permutations and binary trees.

    A Dyck path of length $2n$ is a path from $(0,0)$ to $(n,n)$ with steps $(1,0)$ (up steps) and $(0,1)$ (right steps) that never goes below the diagonal line. There is a bijection between Dyck paths of length $2n$ and binary trees with $n$ nodes as follows. For a binary tree $T$:
    \begin{enumerate}
        \item Create a binary tree $T'$ by adding one child to every node in $T$ that has exactly one child, and adding two children to every node in $T$ that has no child. $T'$ is a \textit{full binary tree}, i.e. a binary tree in which each node has zero or two children, and $T'$ has $2n+1$ nodes. The added nodes are the \textit{leaves} of $T'$, and the original nodes in $T$ are the \textit{internal nodes} of $T'$.
        \item Read $T'$ in \textit{preorder}: first read the root, then read the left subtree in preorder before reading the right subtree in preorder. When we read an internal node, add an up step to the Dyck path. When we read a leaf, add a right step. Note that we always ignore the final leaf since there are $2n+1$ nodes in $T'$ but only $2n$ steps in the Dyck path.
    \end{enumerate}
    Recall that a valley in a Dyck path is a rightstep followed by an upstep. Observe that in the above bijection, the number of right edges in $T$ is the same as the number of valleys in the corresponding Dyck path. For example, in Figure \ref{fig:bt-dyck-bij}, the binary tree has 2 right edges and the corresponding Dyck path has 2 valleys.

    \begin{figure}[h!]
        \centering
        \includegraphics[scale=1]{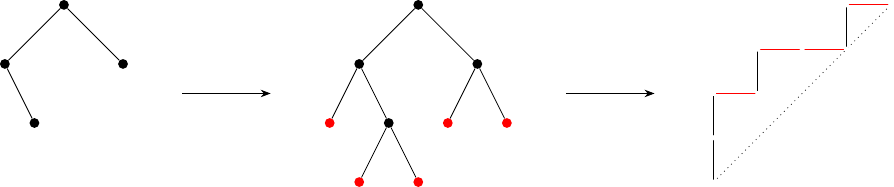}
        \caption{Example of the bijection between binary trees and Dyck paths}
        \label{fig:bt-dyck-bij}
    \end{figure}

    The Catalan convolution is defined as follows.

    \begin{definition}\label{def:cat-con}
        For $n,k \in \ZZ_{\geq 0}$, the $k$th Catalan convolution is
        \[ C_n^{(k)} = \sum_{\substack{a_1+a_2+\ldots+a_{k+1} = n \\ a_1,a_2,\ldots,a_{k+1} \in \ZZ_{\geq 0}}} C_{a_1}C_{a_2}\ldots C_{a_{k+1}}. \]
    \end{definition}

    The explicit formula for $C_n^{(k)}$ is
    \[ C_n^{(k)} = \dfrac{k+1}{n+k+1}\binom{2n+k}{n}. \]
    By definition, $C_n^{(0)} = C_n$ and $C_n^{(1)} = C_{n+1}$. Also, for all $k$, we have $C_0^{(k)} = 1$ and $C_1^{(k)} = k+1$. We will use the following combinatorial interpretation of Catalan convolution: $C_n^{(k)}$ counts the number of Dyck paths of length $2(n+k)$ that start with at least $k$ up steps. To see that this is the correct interpretation, recall that a Dyck path of length $2(n+k)$ starting with at least $k$ up steps corresponds to a parenthesization of $n+k$ pairs of parentheses starting with at least $k$ open brackets. We mark these open brackets. For each marked open bracket, we mark the close bracket that matches it. This gives $k$ marked close brackets. In the Dyck path, we mark the steps corresponding to the marked brackets. Thus, the Dyck path has the following form:
    \[ \underbrace{\textcolor{blue}{U},\ldots,\textcolor{blue}{U},\textcolor{blue}{U}}_{k~\text{up steps}}, D_1, \textcolor{blue}{R}, D_2, \textcolor{blue}{R}, \ldots, D_k, \textcolor{blue}{R}, D_{k+1}, \]
    where the marked steps are colored blue. Observe that the steps in $D_1$ correspond to the brackets inside the inner-most pair of marked brackets. These brackets have to form a parenthesization. Thus, $D_1$ is a Dyck path of length $2a_1 \geq 0$. Similarly, each $D_i$ is a Dyck path of length $2a_i \geq 0$. Note that some $D_i$ may have length zero, so we may have consecutive marked right steps.

    \begin{definition}\label{def:D_n,k}
        For $n,k \geq 0$, we define $\DDD_{n,k}$ to be the set of all Dyck paths of length $2(n+k)$ that start with $k$ up steps. For each Dyck path $D \in \DDD_{n,k}$, let $c(D)$ be the vector where $c_i$ is the length of the $i$th block of consecutive marked right steps.
    \end{definition}
    
    Thus, $c(D)$ is a composition of $k$ and also depends on $k$. For example, Figures \ref{subfig:c(D)-ex1} and \ref{subfig:c(D)-ex2} both show the same Dyck path $D$. However, in Figure \ref{subfig:c(D)-ex1}, we view $D$ as an element of $\DDD_{5,4}$, so $c(D) = (1,2,1)$, which is a composition of $4$. In Figure \ref{subfig:c(D)-ex2}, we view $D$ as an element of $\DDD_{6,3}$, so $c(D) = (2,1)$, which is a composition of $3$.

    \begin{figure}[h!]
     \centering
        \begin{subfigure}[c]{0.45\textwidth}
            \centering
            \includegraphics[scale = 0.7]{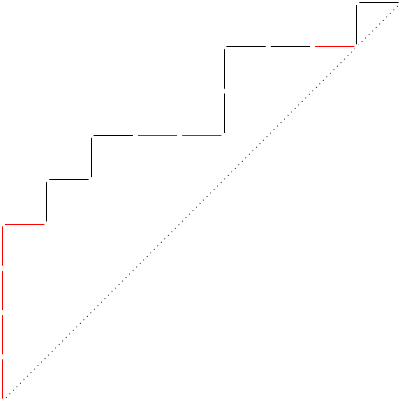}
            \caption{A Dyck path in $\DDD_{5,4}$}
            \label{subfig:c(D)-ex1}
        \end{subfigure}
     \quad
        \begin{subfigure}[c]{0.45\textwidth}
            \centering
            \includegraphics[scale = 0.7]{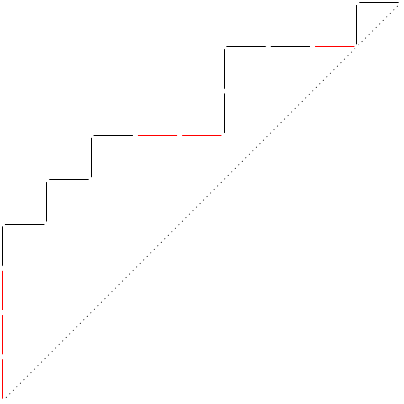}
            \caption{A Dyck path in $\DDD_{6,3}$}
            \label{subfig:c(D)-ex2}
        \end{subfigure}

        \caption{The same Dyck path but viewed as an element of two different sets}
        \label{fig:c(D)-ex}
    \end{figure}
    
    Given a permutation $w \in \SSS_k$ and a composition $c = (c_1,\ldots,c_\ell)$ of $k$, $c$ divides the indices $1,2,\ldots,k$ into $\ell$ blocks: the first block consists of the indices $1,2,\ldots,c_1$; the second block consists of the indices $c_1+1,c_1+2,\ldots,c_1+c_2$; and so on. We define the \textit{descent of $w$ with respect to $c$} as
    \[ \des_c(w) = |\{i~|~\text{$i$ and $i+1$ are in the same block divided by $c$ and}~ w_i > w_{i+1}\}|. \]
    For example, $\des_{(2,2)}(4312) = 1$ because even though $w_2 > w_3$, $2$ and $3$ are not in the same block divided by $(2,2)$, so this descent does not count.
    
    \begin{definition}\label{def:P_n,k}
        For $n,k\geq 0$, we define
        \[ \PPP_{n,k} = \{(w,D)~|~ w \in \SSS_k, D\in \DDD_{n,k}\}. \]
        For each pair $(w,D) \in \PPP_{n,k}$, we define
        \[ \des(w,D) = \des_{c(D)}(w) + \#\text{valley in $D$}. \]
    \end{definition}

\section{Stack-sorting}\label{sec:stack-sorting}

    Recall that we defined $\SSS_{n,k} = \{w~|~w\in \SSS_{n+k}, w_i = i~\text{for all}~i>k\}$. In this section, we will construct a bijection between $s^{-1}(\SSS_{n,k})$ and $\PPP_{n,k}$ that preserves the number of descents.
    
    For $w\in s^{-1}(\SSS_{n,k})$, let $T = \mathcal{I}^{-1}(w)$ be the decreasing binary tree corresponding to $w$. Define the \textit{core tree} of $T$ to be the induced subtree of $T$ form by the nodes $k+1,\ldots,n+k$. Note that the core tree of $T$ is connected, and the nodes have to be labeled from $k+1$ to $n+k$ in postorder. A node in the core tree is \textit{marked} if it contains node $k+1$ in its right subtree. Let $a_1,\ldots,a_{\ell-1}$ be the marked nodes. Observe that since they all contain $k+1$ in their right subtree, they are totally ordered $k+1 = a_\ell <_T a_{\ell_1} <_T \ldots <_T a_1$. Recall that when reading $T$ in postorder, we obtain a permutation in $\SSS_{n,k}$. In particular, the nodes appearing before node $k+1$ in postorder are exactly the nodes $1,\ldots,k$. Thus, the nodes in $T_{\leq k+1}$ and in the left subtree of the marked nodes are exactly the nodes $1,\ldots,k$.

    Define the sequence $c(T)$ as follows: for $1\leq i \leq \ell$, $c_i$ equals the number of nodes in the left subtree of $a_i$ in $T$; in addition, $c_{\ell + 1}$ equals the number of nodes in the right subtree of $a_\ell = k+1$ in $T$. The nodes in the left subtrees of $a_i$'s and in the right subtree of $k+1$ are exactly the nodes $1,\ldots,k$, so $c(T)$ is a weak composition of $k$.

    For each marked node, we now remove its right edge. This divides the core tree of $T$ into $\ell$ disjoint trees $B_1,\ldots,B_\ell$, with $B_i$ containing $a_i$. Furthermore, in $B_i$, $a_i$ is a leaf, and $a_i$ is the left most node, i.e. the unique path from the root to $a_i$ consists of only left edges. We construct a sequence of Dyck paths $D_1,\ldots,D_\ell$ corresponding to $T$ as follows. For each $B_i$,
    \begin{enumerate}
        \item let $B_i'$ be $B_i \backslash \{a_i\}$;
        \item let $D_i'$ be the Dyck path corresponding to $B'_i$ by the bijection in Section \ref{subsec:catalan-convolution};
        \item let $D_i$ be $U,D_i',R$.
    \end{enumerate}
    Observe that each $D_i$ is a Dyck path that never returns to the diagonal. Furthermore, the total length of these Dyck paths is exactly $2n$ since there are $n$ nodes in the core tree of $T$. Now we are ready to state our bijection.

    \begin{definition}\label{def:ss-bij}
        Define the map
        \[ f_{n,k}: s^{-1}(\SSS_{n,k}) \rightarrow \PPP_{n,k} \]
        as follows. For $w\in s^{-1}(\SSS_{n,k})$, let $T = \mathcal{I}^{-1}(w)$. Let $D_1,\ldots,D_\ell$ be the sequence of Dyck paths corresponding to $T$, and let $c(T) = (c_1,\ldots,c_{\ell+1})$. We have $f_{n,k}(w) = (\omega, D)$, where
        \begin{itemize}
            \item $\omega$ is obtained by removing all numbers $k+1,\ldots,n$ in $w$, and
            \item $D$ has the form
            \[ \underbrace{U,\ldots,U}_{k~\text{up steps}},\underbrace{R,\ldots,R}_{c_{\ell+1}~\text{right steps}}, D_\ell, \underbrace{R,\ldots,R}_{c_\ell~\text{right steps}}, D_{\ell-1}, \ldots, \underbrace{R,\ldots,R}_{c_2~\text{right steps}}, D_1, \underbrace{R,\ldots,R}_{c_1~\text{right steps}}. \]
        \end{itemize}
        Note that another way to get $\omega$ is to read the nodes $1,\ldots,k$ in $T$ in inorder. Furthermore, $c(T)$ is a weak composition of $k$, and the total length of $D_1,\ldots,D_\ell$ is $2n$. Thus, $D$ is a Dyck path of length $2(n+k)$ starting with $k$ up steps, i.e. $D\in \DDD_{n,k}$. Therefore, $f_{n,k}(w)$ is indeed in $\PPP_{n,k}$ since $\omega\in \SSS_k$, and $D \in \DDD_{n,k}$.
    \end{definition}

    Let us show an example of this map. Figure \ref{fig:f_n,k-ex1} shows a binary tree $T$ with $\mathcal{I}(T) \in s^{-1}(\SSS_{11,6})$. The marked nodes of $T$ are colored red, i.e. $a_1 = 13$, $a_2 = 11$, $a_3 = 10$, and $a_4 = 7$. Thus, we have $c_1 = 2$ since there are two nodes in the left subtree of $a_1 = 13$. Similarly, $c_2 = 0$, $c_3 = 1$, $c_4 = 2$. Finally, $c_5 = 1$ since there is one node in the right subtree of $a_4 = 7$.

    \begin{figure}[h!]
        \centering
        \includegraphics[scale=0.7]{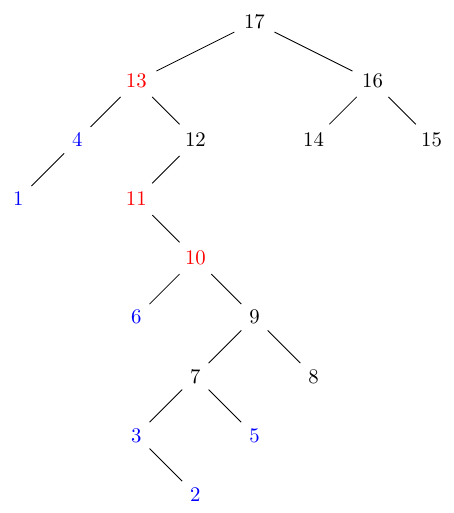}
        \caption{A binary $T$ with $\mathcal{I}(T)\in s^{-1}(\SSS_{11,6})$}
        \label{fig:f_n,k-ex1}
    \end{figure}

    Next, removing the right edges of $a_i$ for $1\leq i < 4$, we obtain four disjoint binary trees shown in Figure \ref{fig:f_n,k-ex2}. Figure \ref{fig:f_n,k-ex2} also shows the corresponding Dyck paths. Observe that these are Dyck paths that never return to the diagonal (until the last step).

    \begin{figure}[h!]
        \centering
        \includegraphics[scale=0.7]{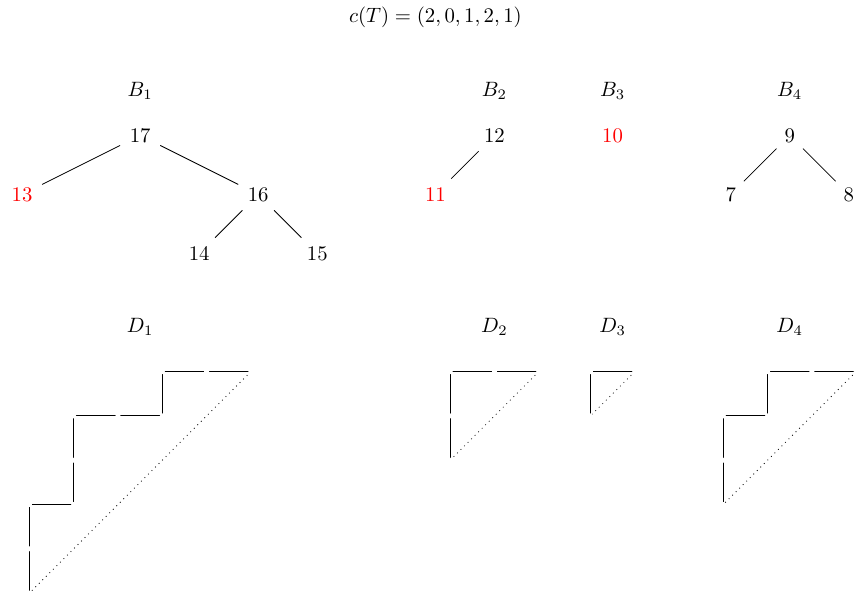}
        \caption{The sequence $c(T)$, the disjoint binary trees and the corresponding Dyck paths}
        \label{fig:f_n,k-ex2}
    \end{figure}

    Putting the Dyck paths and $c(T)$ together, we obtain the Dyck path in Figure \ref{fig:f_n,k-ex3}.

    \begin{figure}[h!]
        \centering
        \includegraphics[scale=0.7]{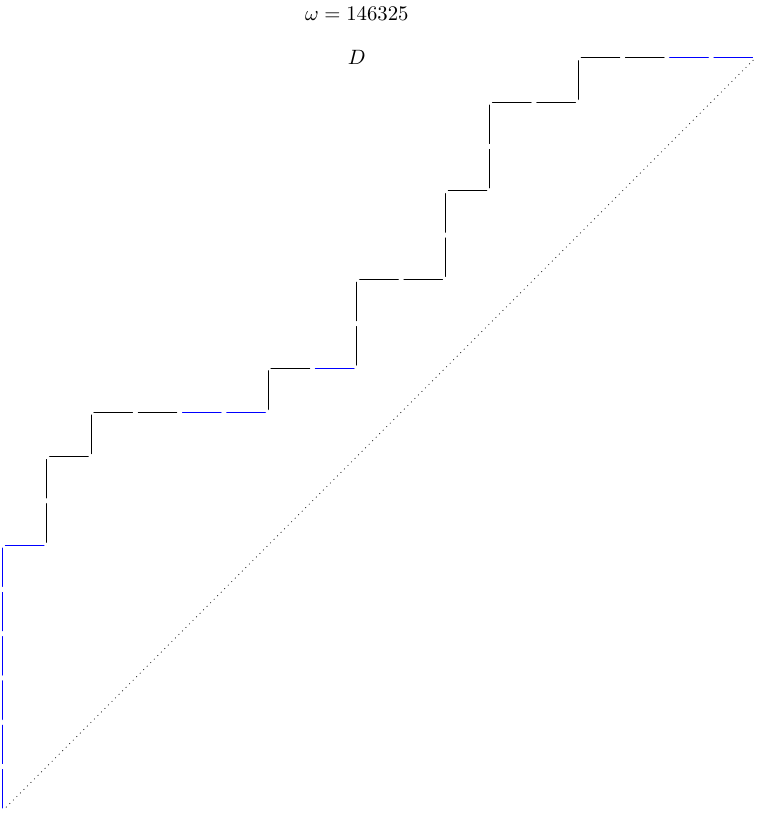}
        \caption{The pair $(\omega, D) = f_{11,6}(T)$}
        \label{fig:f_n,k-ex3}
    \end{figure}

    \begin{prop}\label{prop:ss-bij}
        The map $f_{n,k}$ above is a bijection.
    \end{prop}

    \begin{proof}
        The key here is that the Dyck paths $D_1,\ldots,D_\ell$ in Definition \ref{def:ss-bij} are Dyck paths that never return to the diagonal. This allows us to have a well-defined inverse map. For our convenience, for a Dyck path $D$, let $D_{[i,j]}$ be the segment of $D$ for the $i$th step to the $j$th step. Also, let
        \[ L_D(i) =~\text{\#up steps in $D_{[1,i]}$}~-~\text{\#right steps in $D_{[1,i]}$}. \]
        Since $D$ is a Dyck path, $L_D(i) \geq 0$ for all $i$.

        Given a pair $(\omega,D) \in \PPP_{n,k}$, we now construct its preimage $w = f^{-1}(\omega,D)$. First, we construct a sequence of Dyck paths $D_1,\ldots,D_\ell$ and a sequence $c(T) = (c_1,\ldots,c_{\ell+1})$ as follows.

        \begin{enumerate}
            \item Set $i = k$, $m = \ell+1$.
            \item Let $j$ be the the largest index such that $D_{[i+1,j]} = (R,R,...,R)$. If no such $j$ exists, set $c_m = 0$. If such $j$ exists, set $c_m = j-i$ and set $i = j$.
            \item Set $m = m-1$.
            \item Let $j > i$ be the smallest index such that $L(j) = L(i)$. Note that $j$ must exists since $L(i) \geq 0$ and $L(2(n+k)) = 0$. Set $D_m = D_{[i+1,j]}$ and set $i = j$.
            \item If $i = 2(n+k)$, stop. Else, return to Step 2.
        \end{enumerate}
    
        Note that in Step 4 above, each $D_m$ is indeed a Dyck path since $L(i) = L(j)$. Furthermore, because of the choice of $j$, each $D_m$ is a Dyck path that never returns to the ground level. Next, we construct a sequence of binary trees $B_1,\ldots,B_\ell$ as follows. For each $D_i$,
        \begin{enumerate}
            \item since $D_i$ never returns to the ground level, we can write $D_i = U,D_i',R$ where $D_i'$ is a Dyck path;
            \item let $B_i'$ be the Dyck path corresponding to $D'_i$ by the bijection in Section \ref{subsec:catalan-convolution};
            \item construct $B_i$ by adding a leaf to the left of the left-most node in $B_i'$.
        \end{enumerate}

        Let $a_i$ be the leaf added to $B_i'$ in Step $3$ above. For each $1\leq i \leq \ell-1$, we add $B_{i+1}$ as the right subtree of $a_i$. Label the resulting tree from $k+1$ to $n+k$ in postorder. This gives the core tree of $T = \mathcal{I}^{-1}(w)$.
        
        Next, let $T_1$ be the decreasing tree corresponding to $\omega_1\ldots\omega_{c_1}$. Note that if $c_1 = 0$, $T_1$ is simply an empty tree. We add $T_1$ as the left subtree of $a_1$. Similarly, for $2\leq i \leq \ell$, let $T_i$ be the decreasing tree corresponding to $\omega_{c_1+\ldots+c_{i-1}+1}\ldots\omega_{c_1+\ldots+c_{i}}$. We add $T_i$ as the left subtree of $a_i$. Finally, let $T_{\ell+1}$ be the decreasing tree corresponding to $\omega_{c_1+\ldots+c_{\ell}+1}\ldots\omega_{c_{k}}$. We add $T_{\ell+1}$ as the right subtree of $a_\ell$.

        One can check that the above process is the exact inverse of the map $f_{n,k}$. Furthermore, in this process, every step is well-defined, i.e. we do not have to make any choices. Thus, $f_{n,k}$ is a bijection.
    \end{proof}

    \begin{cor}\label{cor:ss-size}
        For all $n,k \geq 0$, we have
        \[ |s^{-1}(\SSS_{n,k})| = k! \cdot C_n^{(k)}. \]
    \end{cor}

    Recall that $C^{(0)}_k = C_n$. Thus, setting $k = 0$ in Corollary \ref{cor:ss-size}, we recover the well-known result that $|s^{-1}(12\ldots n)| = C_n$.

    \begin{prop}\label{prop:ss-bij-des}
        For any $w\in s^{-1}(\SSS_{n,k})$, we have
        \[ \des(w) = \des(f(w)). \]
    \end{prop}

    \begin{proof}
        Let $f(w) = (\omega,D)$, $T = \mathcal{I}^{-1}(w)$, and let $D_1,\ldots,D_\ell$ be the Dyck paths corresponding to $T$. Also, let $a_1,\ldots,a_{\ell-1}$ be the marked nodes in the construction at the beginning of the section. Let $B_1,\ldots,B_\ell$ be the disjoint trees obtained by  of $T$ corresponding to $D_1,\ldots,D_\ell$, respectively. Recall from Section \ref{subsec:catalan-convolution} that $\des(w)$ equals the number of right edges in $T$. First, we will show that the valleys in $D$ correspond to the right edges incident to nodes $k+1,\ldots,n+k$ in $T$. There are three types of such right edges: those in the subtrees $B_i$'s, the right edges of $a_i$'s, and the right edge of node $k+1$.

        For the first type, recall that for each $i$, $D_i = U,D_i',R$ where $D_i'$ is the Dyck path corresponding to $B_i \backslash \{a_i\}$ through the bijection in Section \ref{subsec:catalan-convolution}. By the definition of that bijection, the number of valleys in $D_i'$ is the same as the number of right edges in $B_i \backslash \{a_i\}$. Also recall that $a_i$ is the leftmost leaf in $B_i$, so adding $a_i$ back to $B_i$ does not create more right edges. Similarly, adding an up step before $D_i'$ and a right step after does not create more valleys. Thus, the number of right edges in $B_i$ is the same as the number of valleys in $D_i$.

        For the second type, there are exactly $\ell-1$ such right edges, one of each $a_i$'s for $1\leq i \leq \ell-1$. In $D$, there are also exactly $\ell - 1$ valleys between $D_{i+1}$ and $D_{i}$ for $1\leq i \leq \ell -1$. Finally, for the third type, the right edge of node $k+1$ exists if and only if there are some nodes in the right subtree of $k+1$, i.e. $c_{\ell+1}$ is positive. In $D$, there is one more valley before $D_\ell$ if and only if $c_{\ell+1}$ is positive.

        Now we prove that the number of right edges that are only incident to nodes $1,\ldots,k$ in $T$ is the same as $\des_{c(D)}(\omega)$. Recall that $\omega$ can be obtained from $T$ by reading the nodes $1,\ldots,k$ in inorder. Therefore, the right edges in the left subtree of $a_1$ correspond to the descents in $\omega_1\ldots\omega_{c_1}$. Similarly, the right edges in the left subtree of $a_i$ correspond to the descents in $\omega_{c_{i-1}+1}\ldots\omega_{c_i}$. Finally, the right edges in the right subtree of $a_\ell = k+1$ correspond to the descents in $\omega_{c_{\ell}+1}\ldots\omega_{c_{\ell+1}}$. The number of such descents is exactly $\des_{(c_1,\ldots,c_{\ell+1})}(\omega) = \des_{c(D)}(\omega)$. This completes the proof.
    \end{proof}

\section{Poset associahedra}\label{sec:poset-ass}

    Now we turn our attention to face numbers of poset associahedra. Recall that the \textit{ordinal sum} of two posets $(P,<_P)$ and $(Q,<_Q)$ is the poset $(R,<_R)$ whose elements are those in $P \cup Q$, and $a\leq_R b$ if and only if
    \begin{itemize}
        \item $a,b\in P$ and $a\leq_P b$, or
        \item $a,b\in Q$ and $a\leq_Q b$, or
        \item $a\in P$ and $b\in Q$.
    \end{itemize}
    We denote the ordinal sum of $P$ and $Q$ as $P~\oplus~Q$. Let $C_n$ be the chain poset of size $n$ and $A_k$ be the antichain of size $k$. Let $A_{n,k} = C_{n+1} \oplus A_k$.

    An $(n,k)$-lollipop graph, denoted $L_{n,k}$, is a graph consisting of a path graph of size $n$ and a complete graph of size $k$, connected by an edge. We call the unique vertex in the complete graph that is adjacent to the path graph the \textit{link vertex}. We call the other vertices in the complete graph the \textit{clique vertices}. We call the other vertices in the path graph the \textit{path vertices}. For instance, in Figure \ref{fig:A-L-ex}, the link vertex is colored blue, and the clique vertices are colored red.

    \begin{figure}[h!]
        \centering
        \includegraphics[scale=0.7]{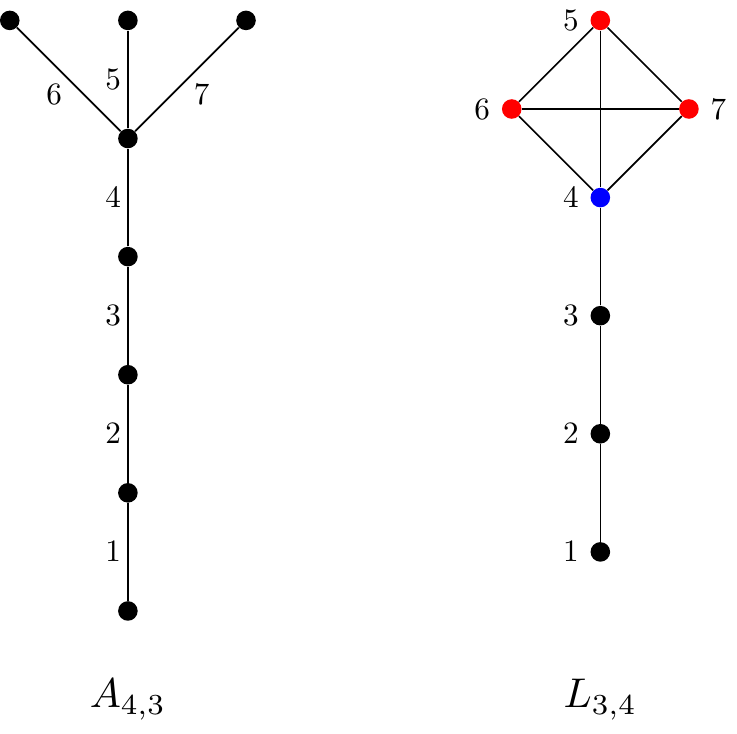}
        \caption{Poset $A_{4,3}$ and graph $L_{3,4}$}
        \label{fig:A-L-ex}
    \end{figure}
    
    Observe that the line graph of the Hasse diagram of $A_{n,k}$ is $L_{n-1,k+1}$. For example, Figure \ref{fig:A-L-ex} shows the correspondence between the edges of the Hasse diagram of $A_{4,3}$ and the vertices of $L_{3,4}$. Recall from the discussion in Section \ref{subsubsec:graph_ass_poset_ass} that this means $\AAA(A_{n,k})$ is isomorphic to $\Ass(L_{n-1,k+1})$. Therefore, instead of studying the $h$-vector of $\AAA(A_{n,k})$, we will study the $h$-vector of $\Ass(L_{n-1,k+1})$.

    Let $\BBB_{n-1,k+1}$ be the set of $\BB$-trees of $L_{n-1,k+1}$. We will construct a descent-preserving bijection between $\BBB_{n-1,k+1}$ and $\PPP_{n,k}$. First, we will label the vertices in $L_{n-1,k+1}$ as follows. We label the link vertex $n$. We label the clique vertices $n+1,\ldots,n+k$. Finally, we label the path vertices $n-1,\ldots,1$ in decreasing order starting from vertex $n$. Figure \ref{fig:tubing-to-B} shows an example of this labeling for $L_{11,4}$.

    \begin{figure}[h!]
        \centering
        \includegraphics[scale = 1.5]{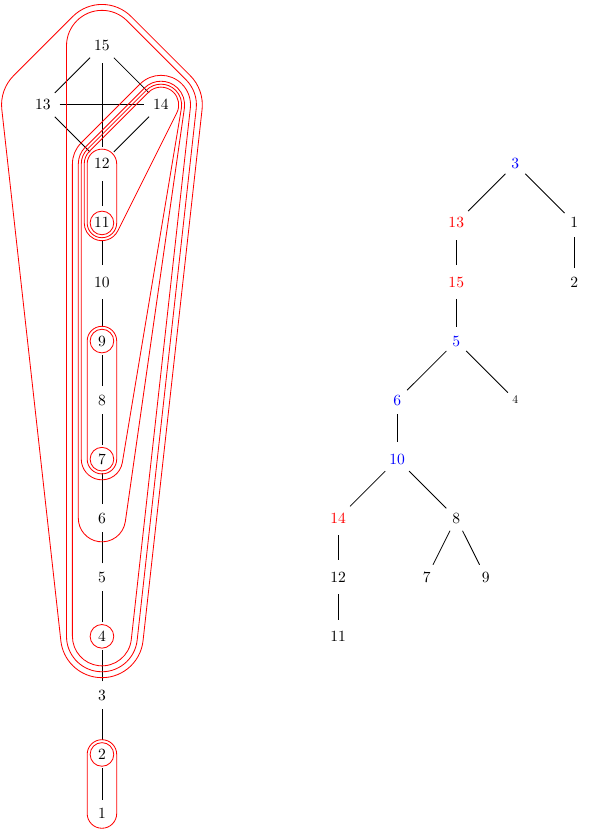}
        \caption{A tubing of $L_{11,4}$ and the corresponding $\BB$-tree}
        \label{fig:tubing-to-B}
    \end{figure}

    First, let us make some observations about the $\BB$-trees in $\BBB_{n-1,k+1}$. Our running example throughout these observations will be Figure \ref{fig:tubing-to-B}.

    \begin{lemma}\label{lem:clique-comparable}
        Let $B\in \BBB_{n-1,k+1}$. The nodes $n+1,\ldots,n+k$ are totally ordered.
    \end{lemma}

    \begin{proof}
        Suppose we have two incomparable nodes $i,j \in \{ n+1,\ldots,n+k \}$, then $T_{\leq i}$ and $T_{\leq j}$ are two disjoint tubes. Their union is also a tube since $i$ and $j$ are adjacent vertices. This is a contradiction.
    \end{proof}

    Lemma \ref{lem:clique-comparable} means that we have a chain $w_1 <_T w_2 <_T \ldots <_T w_k$ where $w_i \in \{n+1,\ldots,n+k\}$. We call the unique path from $w_1$ to the root the \textit{core chain} of $B$. Let $a_1 <_T a_2 <_T \ldots <_T a_\ell$ be the other nodes in the core chain. For example, in Figure \ref{fig:tubing-to-B}, the clique nodes (colored red) are totally ordered $14 <_T 15 <_T 13$. The other elements of the core chain are colored blue.

    \begin{lemma}\label{lem:clique-one-child}
        The nodes $n+1\ldots,n+k$ have at most one child.
    \end{lemma}

    \begin{proof}
        Suppose some node $i\in\{n+1,\ldots,n+k\}$ has two children $i_1,i_2$. Then $T_{\leq i_1}$ and $T_{\leq i_2}$ both have nodes that are adjacent to $i$. However, all nodes that are adjacent to $i$ are pairwise adjacent. This is a contradiction since $T_{\leq i_1}$ and $T_{\leq i_2}$ are two disjoint tubes.
    \end{proof}

    Lemma \ref{lem:clique-one-child} means that in the core chain of $B$, the only nodes that may have two children are $a_1,\ldots,a_\ell$. For these nodes, we call the branch that contains $w_1$ their \textit{main branch}. We call the other branch, if exists, their \textit{secondary branch}. For instance, in Figure \ref{fig:tubing-to-B}, the clique nodes all have one child. The other nodes in the core chain may or may not have two children. For node 10, which has two children, the secondary branch consists of the nodes $7,8,9$.

    \begin{lemma}\label{lem:a_i-branch}
        We have $a_1 > a_2 > \ldots > a_\ell$. Moreover, the secondary branch of $a_i$ contains exactly the nodes $a_{i+1}+1, a_{i+1}+2, \ldots, a_i-1$.
    \end{lemma}

    \begin{proof}
        First, let us show that $a_i > a_{i+1}$. Observe that $T_{\leq a_i}$ is a tube that contains both $a_{i}$ and $w_1$. Thus, it has to contain all the nodes $a_i, a_i + 1, \ldots, n$. This means that $a_{i+1}$ cannot be any of $a_i,a_i+1,\ldots, n$. Hence, $a_i > a_{i+1}$.

        For the second statement, note that $T_{\leq a_i}$ contains all the nodes $a_i, a_i + 1, \ldots, n$. The main branch of $a_i$ contains $w_1$, so it has to also contain $n$. This is because otherwise the two branches of $a_i$ are two disjoint tubes that contains adjacent vertices $w_1$ and $n$, which is impossible. Similarly, since the main branch of $a_i$ contains $n$ it has to also contain $n-1$. Repeating this argument, the main branch of $a_i$ all of $n,n-1,\ldots,a_i+1$.

        Applying the same argument to $a_{i+1}$, we have the main branch of $a_{i+1}$ also contains all of $n, n-1,\ldots, a_{i+1} + 1$. In the main branch of $a_{i+1}$, the unique path from $a_i$ to $a_{i+1}$ consists only of clique nodes, which only have one child. Thus, $T_{\leq a_i}$ contains all of $n, n-1,\ldots, a_{i+1} + 1$.

        On the other hand, the main branch of $a_i$ cannot contain any node $j < a_i$. This is because otherwise the main branch is a tube that has $w_1$ and $j$ but not $a_i$. Such tube is disconnected, which is impossible. For the same reason, the main branch of $a_{i+1}$ cannot contain any node $j < a_{i+1}$. Similar to above, this means that $T_{\leq a_i}$ cannot contain any node $j < a_{i+1}$.

        Therefore, the secondary branch of $a_i$ has to contain all of $a_{i+1}+1,\ldots,a_i - 1$. Furthermore, by the above argument, no other nodes can be in $T_{\leq a_i}$. Hence, the secondary branch of $a_i$ contains exactly the nodes $a_{i+1}+1,\ldots,a_i - 1$.
    \end{proof}

    For example, in Figure \ref{fig:tubing-to-B}, the secondary branch of $10$ consists of nodes $7,8,9$, which are exactly the numbers between $a_1 = 10$ and $a_2 = 6$.

    \begin{lemma}\label{lem:T_<w1}
        If $w_1$ has a child, then $T_{< w_1} = \{n,n-1,\ldots,a_1+1\}$.
    \end{lemma}

    \begin{proof}
        From the proof of Lemma \ref{lem:a_i-branch}, the main branch of $a_1$ contains all of $n,n-1,\ldots,a_1 + 1$. Moreover, the unique path from $w_1$ to $a_1$ consists only of clique nodes, which only have one child. Thus, $T_{< w_1} = \{n,n-1,\ldots,a_1+1\}$.
    \end{proof}

    Back to our running example, in Figure \ref{fig:tubing-to-B}, the descendants of $w_1 = 14$ are $11$ and $12$, which are exactly the numbers from $n = 12$ to $a_1+1 = 11$.

    Lemmas \ref{lem:a_i-branch} and \ref{lem:T_<w1} means that the descendants of $w_1$ form a $\BB$-tree $B_0$ of the subgraph $(a_1+1) - (a_1 + 2) - \ldots - n$. This subgraph is a path graph of $ n - a_1$ elements. Similarly, the secondary branch of each $a_i$ forms a $\BB$-tree $B_i$ of the subgraph $(a_{i+1} + 1) - (a_{i+1} + 2) - \ldots - (a_{i}-1)$. This is also a path graph of $a_{i} - a_{i+1} -1$ elements (with $a_{\ell + 1} = 0$).

    In \cite[Section 10.2]{postnikov2006faces}, it is shown that there is a bijection between $\BB$-trees of path graphs and binary trees. Moreover, the descent edges of the $\BB$-trees correspond to the right edges of the binary trees. This means that there is a bijection between $\BB$-trees of path graphs and Dyck paths such that the descent edges of the $\BB$-trees correspond to the valleys of the Dyck paths.

    Now we are ready to state our bijection. Let $B_0$ be the tree formed by the descendants of $w_1$. For $1\leq i \leq \ell$, let $B_i$ be the tree formed by the secondary branch of $a_i$. Next, we construct a sequence of Dyck paths $D_1,\ldots,D_\ell$ as follows. For each $B_i$ with $1\leq i \leq \ell$,

    \begin{itemize}
        \item let $D_i'$ be the Dyck path corresponding to $B_i$ by the bijection above;
        \item let $D_i$ be $U,D_i',R$.
    \end{itemize}

    Once again, each $D_i$ is a Dyck path that never returns to the diagonal. Finally, let $D_0$ be the Dyck path corresponding to $B_0$. $D_0$ is a Dyck path that may return to the diagonal. Furthermore, for $1\leq i \leq \ell$, $D'_i$ is a Dyck path of length $(a_i - a{i+1} - 1)$, so $D_i$ is a Dyck path of length $2(a_i - a_{i+1})$. $D_0$ is a Dyck path of length $2(n-a_1)$. Thus, the total length of these Dyck paths is exactly $2n$. Now we are ready to state our bijection.

    \begin{definition}\label{def:bt-bij}
        Define the map
        \[ g_{n,k}: \BBB_{n,k} \rightarrow \PPP_{n,k} \]
        as follows. For $B\in \BBB_{n,k}$, we construct $w_1,\ldots,w_k$ and $a_1,\ldots,a_\ell$ as above. Let $D_0, D_1,\ldots,D_\ell$ be the sequence of Dyck paths constructed as above. Also, for $1< i \leq \ell$, let $c_i$ be the number of clique nodes between $a_i$ and $a_{i-1}$. Let $c_1$ be the number of clique nodes below $a_1$ and $c_{\ell+1}$ be the number of clique nodes above $a_\ell$. We have $g_{n,k}(B) = (w, D)$, where
        \begin{itemize}
            \item $w = (w_1-n),(w_2-n),\ldots, (w_k-n)$, and
            \item $D$ has the form
            \[ \underbrace{U,\ldots,U}_{k~\text{up steps}}, D_0, \underbrace{R,\ldots,R}_{c_{1}~\text{right steps}}, D_1, \underbrace{R,\ldots,R}_{c_2~\text{right steps}}, D_2, \ldots, D_{\ell-1}, \underbrace{R,\ldots,R}_{c_\ell~\text{right steps}}, D_\ell, \underbrace{R,\ldots,R}_{c_{\ell+1}~\text{right steps}}. \]
        \end{itemize}
        By definition, $c_1 + \ldots + c_{\ell+1}$ is the total number of clique nodes, which is $k$. The total length of $D_1,\ldots,D_\ell$ is $2n$. Thus, $D$ is a Dyck path of length $2(n+k)$ starting with $k$ up steps, i.e. $D\in \DDD_{n,k}$. Clearly, $w \in \SSS_k$. Therefore, $g_{n,k}(B)$ is indeed in $\PPP_{n,k}$.
    \end{definition}

    \begin{prop}\label{prop:bt-bij}
        The map $g_{n,k}$ above is a bijection.
    \end{prop}

    \begin{proof}
        Again, the key here is that the Dyck paths $D_1,\ldots,D_\ell$ are Dyck paths that never return to the diagonal. Thus, these can be recovered uniquely similar to in Proposition~\ref{prop:ss-bij}. $D_0$ is a Dyck path that may return to the diagonal, but it is the Dyck path that occurs before the first marked right step. Thus, $D_0$ can also be recovered uniquely. The rest of the inverse map can be constructed similar to Proposition \ref{prop:ss-bij}. We leave the details to the reader.
    \end{proof}

    \begin{prop}\label{prop:bt-bij-des}
        For any $B\in \BBB_{n,k}$, we have
        \[ \des(B) = \des(g(B)). \]
    \end{prop}

    \begin{proof}
        Let $g(B) = (w,D)$. First, observe that by Lemma \ref{lem:a_i-branch}, the edge connecting $a_i$ to its secondary branch is never a descent edge. Similarly, by Lemma \ref{lem:T_<w1}, the edge connecting $w_1$ to its child is also never a descent edge. Thus, there are only two types of descent edges in $B$: those in the subtrees $B_i$'s, and the edges of the core chain.

        Recall that we have a bijection between $\BB$-trees and Dyck paths that matches the number of descent edges and the number of valleys. Thus, for the first type, for $1\leq i\leq \ell$, the number of descent edges in $B_i$ is the same as the number of valleys in $D'_i$. This is also the same as the number of valleys in $D_i = U, D'_i, R$. Finally, the number of descent edges in $B_0$ is the same as the number of valleys in $D_0$.

        For the second type, along the core chain, observe that for each $a_i$, its child in the main branch is either a clique node or $a_{i-1}$. In both cases, this child is greater then $a_i$, so the edge connecting $a_i$ to its main branch is always a descent edge. Respectively, in $D$, there is always a valley before $D_i$ for $1\leq i \leq \ell$.

        The only other descent edges are those connecting clique nodes. The number of such descent edges is exactly $\des_{(c_1,\ldots,c_{\ell+1})}(\omega) = \des_{c(D)}(\omega)$. This completes the proof.
    \end{proof}

    Combining Propositions \ref{prop:ss-bij-des} and \ref{prop:bt-bij}, we have our main theorem.

    \begin{thm}\label{thm:A_n,k-h-vector}
        Let $h = (h_0, h_1, \ldots, h_{n+k-1})$ be the $h$-vector of $\AAA(A_{n,k})$. Then $h_i$ counts the number of permutations in $s^{-1}(\SSS_{n,k})$ with exactly $i$ descents.
    \end{thm}

    We want to point out the following result by Br\"and\'en.

    \begin{thm}[\cite{branden2008actions}]\label{thm:ss-preimage-gamma}
        For $A \subseteq \SSS_n$, we have
        \[ \sum_{\sigma \in s^{-1}(A)} x^{\des(\sigma)} = \sum_{m = 0}^{\lfloor \frac{n-1}{2}\rfloor} \dfrac{|\{ \sigma \in s^{-1}(A)~:~\text{peak}(\sigma) = m \}|}{2^{n-1-2m}} x^m(1+x)^{n-1-2m}, \]
        where $\text{peak}(\sigma)$ is the number of index $i$ such that $\sigma_{i-1} < \sigma_i > \sigma_{i+1}$.
    \end{thm}

    This gives the following corollary.

    \begin{cor}\label{cor:A_n,k-gamma-nonnegative}
        The $\gamma$-vector of $\AAA(A_{n,k})$ is nonnegative.
    \end{cor}

    \begin{remark}
        Corollary \ref{cor:A_n,k-gamma-nonnegative} also follows from \cite[Theorem 11.6]{postnikov2006faces} since the lollipop graph $L_{n-1,k+1}$ is chordal with the labeling we used in this section. Theorem \ref{thm:ss-preimage-gamma} provides another proof of this.
    \end{remark}

\section{Real-rootedness}\label{sec:real-rootedness}

    In this section, we will prove real-rootedness of the $h$-polynomial of $\AAA(A_{n,2})$. We say a polynomial $a_0 + a_1x + \ldots + a_nx^n$ is real-rooted if all of its zeros are real. We say a sequence $(a_0,a_1,\ldots,a_n)$ is real-rooted if its generating function $a_0 + a_1x + \ldots + a_nx^n$ is real-rooted.

    Let $f$ and $g$ be real-rooted polynomials with positive leading coefficients and real roots $\{f_i\}$ and $\{g_i\}$, respectively. We say that $f$ \textit{interlaces} $g$ if
    \[ g_1 \leq f_1 \leq g_2 \leq f_2 \leq \ldots \leq f_{d-1} \leq g_d \]
    where $d = \deg g = \deg f + 1$. We say that $f$ \textit{alternates left of} $g$ if
    \[ f_1 \leq g_1 \leq f_2 \leq g_2 \leq \ldots \leq f_d \leq g_d \]
    where $d = \deg g = \deg f$. Finally, we say $f$ \textit{interleaves} $g$, denoted $f \ll g$, if $f$ either interlaces or alternates left of $g$.

    A classic example of real-rooted polynomials is Narayana polynomials. Recall that the Narayana polynomial $N_n(x)$ is defined by
    \[ N_n(x) = \sum_{i = 0}^{n-1} a_ix^i \]
    where $a_i$ counts the number of permutations in $s^{-1}(12\ldots n)$ with exactly $i$ descents. In other words, $N_n(x)$ is the $h$-polynomial of $\AAA(A_{n,0})$ and $\AAA(A_{n-1,1})$. We have the following result.

    \begin{thm}[\cite{branden2006linear}]\label{thm:N_m-real-rooted}
        For all $n$, $N_n(x)$ is real-rooted. Furthermore, $N_{n-1}(x) \ll N_n(x)$.
    \end{thm}
       
    To prove real-rootedness of the $h$-polynomial of $\AAA(A_{n,2})$, we will need the following ``happy coincidence''.

    \begin{prop}\label{prop:happy-coincidence}
        The number of permutations in $s^{-1}(2134\ldots n)$ with exactly $i$ descents is the same as the number of permutations $w$ in $s^{-1}(1234\ldots n)$ with exactly $i$ descents such that $w_1, w_n < n$.
    \end{prop}

    \begin{proof}
        Let 
        \[\mathcal{T}_1 = \{T~|~\mathcal{I}(T) = w \in s^{-1}(1234\ldots n), w_1,w_n < n \} \]
        and
        \[\mathcal{T}_2 = \{T~|~\mathcal{I}(T) = w \in s^{-1}(2134\ldots n) \}. \]
        We will construct a bijection between $\mathcal{T}_1$ and $\mathcal{T}_2$ that preserves the number of right edges.

        For two nodes $v_1,v_2$ in a binary tree $T$, we say $v_1 \rightarrow_R v_2$ (resp. $\rightarrow_L$) if $v_1$ is the right (resp. left) child of $v_2$. Our bijection $\varphi$ is constructed as follows.

        Given $T \in \mathcal{T}_1$, let $v$ be the smallest ancestor of node $1$ that has two children. Then, we must have a chain $1 \rightarrow_{D_1} 2 \rightarrow_{D_2} \ldots \rightarrow_{D_{v-1}} v$ where each $D_i$ is either $R$ or $L$, and each node $2,3,\ldots,v-1$ has exactly one child. Furthermore, since $\mathcal{I}(T) = 1234\ldots n$ and $v$ has two children, $1$ has to be in the left-subtree of $v$, so $D_{v-1} = L$. Then, $\varphi(T) \in \mathcal{T}_2$ is constructed as follows.

        \begin{enumerate}
            \item Remove all nodes below $v-1$.
            \item The root of $T$ has to be $n$, add the follow edges: $n \rightarrow_{D_{v-2}} n+1 \rightarrow_{D_{v-3}} \ldots \rightarrow_{D_1} n+v-2$.
            \item Relabel the nodes such that the postorder reading word is $2134\ldots n$.
        \end{enumerate}

        Clearly, $\varphi$ preserves the number of right edges: we simply move the right edges from below $v-1$ to above $n$. Also, since $\mathcal{I}(T) = w$ has $w_1, w_n < n$, the root $n$ of $T$ must have two children. Hence, the node $v$ above always exists. Finally, after step 2, the first two nodes in postorder are $v-1$ and the right child of $v$. These two nodes are incomparable, so we can always relabel the nodes in step 3. Therefore, $\varphi$ is well-defined.

        We can also construct the inverse map as follows. Given $T' \in \mathcal{T}_2$, let $v$ be the largest descendant of the root $n$ that has two children. Then, we must have a chain $v \rightarrow_{D_v} v+1 \rightarrow_{D_{v+1}} \ldots \rightarrow_{D_{n-1}} n$ where each $D_i$ is either $R$ or $L$, and each node $v+1,v+2,\ldots,n$ has exactly one child. Furthermore, since $T'\in\mathcal{T}_2$, the nodes $1$ and $2$ in $T'$ are incomparable, so $v \geq 3$. Then, $\varphi^{-1}(T') \in \mathcal{T}_1$ is constructed as follows.

        \begin{enumerate}
            \item Remove all nodes above $v$.
            \item Node $2$ is not above $v$, so it is not removed. Add the follow edges: $n \rightarrow_{D_{n-1}} n-1 \rightarrow_{D_{n-2}} \ldots \rightarrow_{D_{v+1}} v+1 \rightarrow_{D_v} 2$.
            \item Relabel the nodes such that the postorder reading word is $1234\ldots n$.
        \end{enumerate}

        In $T'$, observe that $1$ and $2$ are incomparable. Thus, at least one node in $T'$ must have two children. Hence, the node $v$ above always exists, so $\varphi^{-1}$ is well-defined. One can easily check that $\varphi^{-1}$ is indeed the inverse map of $\varphi$. This completes the proof.
    \end{proof}

    An example of the map $\varphi$ above can be seen in Figure \ref{fig:varphi-ex}.

    \begin{figure}[h!]
        \centering
        \includegraphics{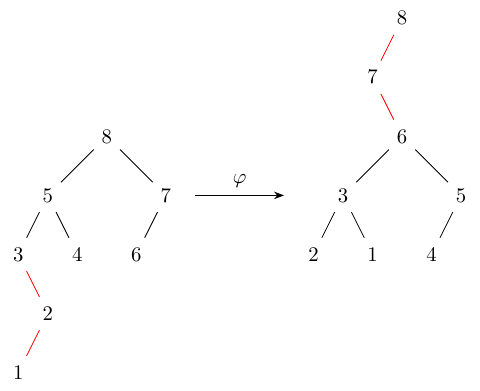}
        \caption{Example of the map $\varphi$}
        \label{fig:varphi-ex}
    \end{figure}

    Proposition \ref{prop:happy-coincidence} gives the following important recurrence.

    \begin{prop}\label{prop:A_2-recurrence}
        Let $H_{n}(x)$ be the $h$-polynomial of $\AAA(A_{n,2})$, and recall that $N_{n+2}(x)$ and $N_{n+1}(x)$ are the $h$-polynomials of $\AAA(A_{n+2,0})$ and $\AAA(A_{n+1,0})$, respectively. We have
        \[ H_{n}(x) = 2N_{n+2}(x) - (1+x)N_{n+1}(x). \]
    \end{prop}

    \begin{proof}
        Recall that $H_{n}(x)$, which is the $h$-polynomial of $\AAA(A_{n,2})$, is the generating function counting descents in $s^{-1}(123\ldots n+2, 213\ldots n+2)$. Thus, we can write
        \[ H_n(x) = N_{n+2}(x) + P_{n+2}(x) \]
        where $N_{n+2}(x)$ counts descents in $s^{-1}(123\ldots n+2)$ and $P_{n+2}(x)$ counts descents in $s^{-1}(213\ldots n+2)$.

        Proposition \ref{prop:happy-coincidence} says that the generating function counting descents in $s^{-1}(213\ldots n+2)$ is exactly the same as that counting descents in
        \[ \{w~|~w\in \text{ss}_{n+2}\} ~\backslash  \left( \{w~|~w\in \text{ss}_{n+2}, w_1 = n+2\} \cup \{w~|~w\in \text{ss}_{n+2}, w_{n+2} = n+2\} \right), \]
        where $\text{ss}_{n+2} = s^{-1}(123\ldots n+2)$. The generating function counting descents in $\{w~|~w\in \text{ss}_{n+2}\}$ is exactly $N_{n+2}(x)$. On the other hand, there is a natural descent-preserving bijection between $s^{-1}(123\ldots n+1)$ and $\{w~|~w\in \text{ss}_{n+2}, w_{n+2} = n+2\}$ by appending $n+2$ to the end of the elements in $s^{-1}(123\ldots n+1)$. Thus, the generating function counting descents in $\{w~|~w\in \text{ss}_{n+2}, w_{n+2} = n+2\}$ is $N_{n+1}(x)$. Similarly, there is a natural bijection between $s^{-1}(123\ldots n+1)$ and $\{w~|~w\in \text{ss}_{n+2}, w_1 = n+2\}$ by appending $n+2$ to the beginning of the elements in $s^{-1}(123\ldots n+1)$. This bijection adds an extra descent at the beginning of every element. Thus, the generating function counting descents in $\{w~|~w\in \text{ss}_{n+2}, w_1 = n+2\}$ is $xN_{n+1}(x)$.

        Hence,
        \[ P_{n+2}(x) = N_{n+2}(x) - (1+x)N_{n+1}(x). \]
        Therefore,
        \[ H_n(x) = N_{n+2}(x) + P_{n+2}(x) = 2N_{n+2}(x) - (1+x)N_{n+1}(x). \]
        This completes the proof.
    \end{proof}

    From the recurrence in Proposition \ref{prop:A_2-recurrence}, we have the following theorem.

    \begin{thm}\label{thm:A_2-real-rooted}
        Let $H_n(x)$ be the $h$-polynomial of $\AAA(A_{n,2})$. Then, $H_n(x)$ is real-rooted.
    \end{thm}

    \begin{proof}
        We have two cases: $n$ is even and $n$ is odd.

        \textbf{Case 1:} $n$ is even.

        In this case, $N_{n+1}(x)$ has $n$ roots $f_1 < f_2 < \ldots < f_n < 0$. Since $n$ is even, $-1$ is not a root of $N_{n+1}(x)$. Let $k$ be the index such that $f_k < -1 < f_{k+1}$. $N_{n+1}(x)$ is a monic polynomial of even degree, so the sign of $N_{n+1}(x)$ alternates in the intervals $(-\infty , f_1)$, $(f_1,f_2),\ldots,(f_{n-1}, f_n)$, and $(f_n, \infty)$ with $N_{n+1}(x) > 0$ in $(-\infty, f_1)$ and $(f_n, \infty)$. 
        Thus, the sign of $-(1+x)N_{n+1}(x)$ alternates in the intervals $(-\infty , f_1)$, $(f_1,f_2),\ldots, (f_k,-1)$, $(-1,f_{k+1}),\ldots,(f_{n-1}, f_n)$, and $(f_n, \infty)$ with $N_{n+1}(x) > 0$ in $(-\infty, f_1)$ and $N_{n+1}(x) < 0$ in $(f_n, \infty)$.

        Recall that $N_{n+2}(x)$ has a root in each of the intervals $(-\infty , f_1)$, $(f_1,f_2),\ldots,(f_{n-1}, f_n)$, and $(f_n, \infty)$. Let the roots of $N_{n+2}(x)$ be $g_1<g_2<\ldots<g_{n+1} < 0$. Then, at each $g_i$, the sign of $H_{n}(x) = 2N_{n+2}(x) - (1+x)N_{n+1}(x)$ is the same as the sign of $- (1+x)N_{n+1}(x)$. Therefore, the sign of $H_n(x)$ alternates at $g_1,\ldots,g_{k-1}$ and also at $g_{k+1},\ldots,g_{n+1}$. Thus, $H_n(x)$ has a root in each of the intervals $(g_1,g_2),\ldots,(g_{k-2},g_{k-1})$ and $(g_{k+1}, g_{k+2}), \ldots, (g_n,g_{n+1})$. Furthermore, since $H_n(x)$ is a monic polynomial of odd degree, $H_n(x) > 0$ as $x\rightarrow\infty$. However, $H_n(g_{n+1}) < 0$, so $H_n(x)$ has another root in $(g_{n+1},\infty)$. Similarly, $H_n(x)$ has another root in $(-\infty,g_{1})$. Finally, 1 is another root of $H_n(x)$ since 1 is a root of both $N_{n+2}(x)$ and $-(1+x)N_{n+1}(x)$. Hence, $H_n(x)$ has exactly $n+1$ real roots, so $H_n(x)$ is real-rooted.

        \textbf{Case 2:} $n$ is odd.

        In this case, $N_{n+1}(x)$ has $n$ roots $f_1 < f_2 < \ldots < f_n < 0$. Since $n$ is odd, $-1$ is a root of $N_{n+1}(x)$. Let $k$ be the index such that $f_k = -1$. $N_{n+1}(x)$ is a monic polynomial of even degree, so the sign of $N_{n+1}(x)$ alternates in the intervals $(-\infty , f_1)$, $(f_1,f_2),\ldots,(f_{n-1}, f_n)$, and $(f_n, \infty)$ with $N_{n+1}(x) < 0$ in $(-\infty, f_1)$ and $N_{n+1}(x) > 0$ in $(f_n, \infty)$. 
        Thus, the sign of $-(1+x)N_{n+1}(x)$ alternates in the intervals $(-\infty , f_1)$, $(f_1,f_2),\ldots, (f_{k-1},f_{k+1}),\ldots,(f_{n-1}, f_n)$, and $(f_n, \infty)$ with $N_{n+1}(x) < 0$ in $(-\infty, f_1)$ and $(f_n, \infty)$.

        Again, $N_{n+2}(x)$ has a root in each of the intervals $(-\infty , f_1)$, $(f_1,f_2),\ldots,(f_{n-1}, f_n)$, and $(f_n, \infty)$. Let the roots of $N_{n+2}(x)$ be $g_1<g_2<\ldots<g_{n+1} < 0$. Then, at each $g_i$, the sign of $H_{n}(x) = 2N_{n+2}(x) - (1+x)N_{n+1}(x)$ is the same as the sign of $- (1+x)N_{n+1}(x)$. Therefore, the sign of $H_n(x)$ alternates at $g_1,\ldots,g_{k}$ and also at $g_{k+1},\ldots,g_{n+1}$. Thus, $H_n(x)$ has a root in each of the intervals $(g_1,g_2),\ldots,(g_{k-1},g_{k})$ and $(g_{k+1}, g_{k+2}), \ldots, (g_n,g_{n+1})$. Furthermore, since $H_n(x)$ is a monic polynomial of even degree, $H_n(x) > 0$ as $x\rightarrow\infty$. However, $H_n(g_{n+1}) < 0$, so $H_n(x)$ has another root in $(g_{n+1},\infty)$. Similarly, $H_n(x)$ has another root in $(-\infty,g_{1})$. Hence, $H_n(x)$ has exactly $n+1$ real roots, so $H_n(x)$ is real-rooted.
    \end{proof}

\section{Further directions}

    We conjecture that the recurrence in Proposition \ref{prop:A_2-recurrence} can be generalized to any poset. Recall that for a poset $P$, a subposet $S$ of $P$ is called \textit{autonomous} if for every element $p \in P \backslash S$, either
    \begin{itemize}
        \item $p > s$ for all $s\in S$, or
        \item $p < s$ for all $s\in S$, or
        \item $p ~||~ s$ for all $s\in S$.
    \end{itemize}

    Our conjecture is as follows.

    \begin{conjecture}\label{con:recurrence}
        Let $P$ be a poset with an autonomous subposet $S$ that is a chain of size $2$, i.e. $S = C_2$. Let $P_1$ be the poset obtained from $P$ by replacing $S$ by a singleton. Let $P_2$ be the poset obtained from $P$ by replacing $S$ by an antichain of size $2$, i.e. $A_2$. Let $h_{P}(x)$, $h_{P_1}(x)$, $h_{P_2}(x)$ be the $h$-polynomials of $\AAA(P)$, $\AAA(P_1)$, $\AAA(P_2)$, respectively. Then,
        \[ h_{P_2}(x) + (1+x)h_{P_1}(x) = 2h_{P}(x). \]
        As a result, let $\gamma_{P}(x)$, $\gamma_{P_1}(x)$, $\gamma_{P_2}(x)$ be the $\gamma$-polynomials of $\AAA(P)$, $\AAA(P_1)$, $\AAA(P_2)$, respectively. Then,
        \[ \gamma_{P_2}(x) + \gamma_{P_1}(x) = 2\gamma_{P}(x). \]
    \end{conjecture}

    Proving the recurrence in Conjecture \ref{con:recurrence} would be useful in proving real-rootedness of the $h$-polynomials, as shown in Theorem \ref{thm:A_2-real-rooted}. Furthermore, the resulting recurrence of the $\gamma$-polynomial would also be useful in proving $\gamma$-positivity. More generally, we have the following conjecture when $S$ is an antichain of size $n$.

    \begin{conjecture}\label{con:recurrence3}
        Let $P$ be a poset with an autonomous subposet $S$ that is a chain of size $n$, i.e. $S = C_3$. For $1\leq i \leq n$, let $P_i$ be the poset obtained from $P$ by replacing $S$ by an antichain of size $i$, i.e. $A_i$. Let $h_{P}(x)$, $h_{P_1}(x)$, $\ldots$, $h_{P_n}(x)$ be the $h$-polynomials of $\AAA(P)$, $\AAA(P_1)$, $\ldots$, $\AAA(P_n)$, respectively. Then,
        \[ \sum_{\substack{c = (c_1,\ldots,c_n) \\ \sum ic_i = n}} \dfrac{n!}{1^{c_1}c_1!\ldots n^{c_n}c_n!}B_{1}(x)^{c_1}\ldots B_{n}(x)^{c_n} h_{P_{\ell(c)}}(x) = n!\cdot h_{P}(x) \]
        where
        \[ B_k(x) = \sum_{i = 0}^{k-1}\binom{k-1}{i}^2x^i \]
        are type B Narayana polynomials and $\ell(c) = \sum_{i=1}^n c_n$.

        In particular, when $n=3$, we have
        \[ h_{P_3}(x) + 3(1+x)h_{P_2}(x) + 2(1+4x+x^2)h_{P_1}(x) = 6h_{P}(x). \]
    \end{conjecture}

    The type B Narayana polynomials above also show up as the rank-generating function of the type B analogue $\text{NC}^B_n$ of the lattice of non-crossing partitions (see \cite{reiner1997non}) and the $h$-polynomials of type B associahedra (see \cite{simion2003type}). In particular, the sum of the coefficients in $B_{n+1}(x)$ is $\binom{2n}{n}$, which is called type B Catalan number.

\section*{Acknowledgements}
    The authors are grateful to Vic Reiner and Pavel Galashin for their guidance. We also thank Colin Defant for helpful discussions about stack-sorting.

\bibliography{bibliography}

\newcommand{\etalchar}[1]{$^{#1}$}
\begin{thebibliography}{PRW06}

\bibitem[B{\'o}n22]{bona2022combinatorics}
Mikl{\'o}s B{\'o}na.
\newblock {\em Combinatorics of permutations}.
\newblock CRC Press, 2022.

\bibitem[Br{\"a}06]{branden2006linear}
Petter Br{\"a}nd{\'e}n.
\newblock On linear transformations preserving the p{\'o}lya frequency
  property.
\newblock {\em Transactions of the American Mathematical Society},
  358(8):3697--3716, 2006.

\bibitem[Br{\"a}08]{branden2008actions}
Petter Br{\"a}nd{\'e}n.
\newblock Actions on permutations and unimodality of descent polynomials.
\newblock {\em European Journal of Combinatorics}, 29(2):514--531, 2008.

\bibitem[Gal21]{galashin2021poset}
Pavel Galashin.
\newblock Poset associahedra.
\newblock {\em arXiv preprint arXiv:2110.07257}, 2021.

\bibitem[K{\etalchar{+}}73]{knuth1973art}
Donald~Ervin Knuth et~al.
\newblock {\em The art of computer programming}, volume~3.
\newblock Addison-Wesley Reading, MA, 1973.

\bibitem[PRW06]{postnikov2006faces}
Alexander Postnikov, Victor Reiner, and Lauren Williams.
\newblock Faces of generalized permutohedra.
\newblock {\em arXiv preprint math/0609184}, 2006.

\bibitem[Rei97]{reiner1997non}
Victor Reiner.
\newblock Non-crossing partitions for classical reflection groups.
\newblock {\em Discrete Mathematics}, 177(1-3):195--222, 1997.

\bibitem[Sim03]{simion2003type}
Rodica Simion.
\newblock A type-b associahedron.
\newblock {\em Advances in Applied Mathematics}, 30(1-2):2--25, 2003.

\bibitem[Wes90]{west1990permutations}
Julian West.
\newblock {\em Permutations with forbidden subsequences, and, stack-sortable
  permutations}.
\newblock PhD thesis, Massachusetts Institute of Technology, 1990.

\end{thebibliography}
\bibliographystyle{alpha}

\end{document}